\documentclass[11pt,reqno]{amsart}
\usepackage{graphicx}
\usepackage[english]{babel}
\usepackage{amssymb,enumerate,bbm,amsmath}
\usepackage{url}
\usepackage[linktocpage=true,colorlinks=true,linkcolor=blue,citecolor=magenta,urlcolor=blue]{hyperref}
\usepackage{pgfplots}
\usepackage{tikz}
\usetikzlibrary{arrows,shapes,trees,backgrounds}
\usepackage[margin=1.1in]{geometry}
\usepackage{comment}
\newcommand{\MM}{\mathcal{M}}
\newcommand{\RR}{\mathbb{R}}
\newcommand{\HH}{\mathbb{H}}
\newcommand{\SSS}{\mathbb{S}}

\newcommand{\p}{\partial}

\newtheorem{prop}{Proposition}[section]
\newtheorem{thm}[prop]{Theorem}
\newtheorem{lem}[prop]{Lemma}
\newtheorem{cor}[prop]{Corollary}
\newtheorem{defn}[prop]{Definition}

\theoremstyle{remark}
\newtheorem{rem}[prop]{Remark}
\newtheorem{ques}[prop]{Question}
\newtheorem*{ack}{Acknowledgments}

\newcommand{\lemref}[1]{Lemma \ref{#1}}
\newcommand{\corref}[1]{Corollary \ref{#1}}
\newcommand{\propref}[1]{Proposition \ref{#1}}

\numberwithin{equation}{section}
\allowdisplaybreaks
\begin{document}
\title[On a Class of Fully Nonlinear Curvature Flows in Hyperbolic Space]{On a Class of Fully Nonlinear Curvature Flows in Hyperbolic Space}

\author[F.Hong]{Fang Hong}
\address{School of Gifted Young, University of Science and Technology of China, Hefei, 230026,
	P. R. China}
\email{\href{mailto:hongf@mail.ustc.edu.cn}{hongf@mail.ustc.edu.cn}}

\subjclass[2010]{53C44, 53C21}
\keywords{hyperbolic space, fully nonlinear curvature flow, $k$-curvature}

\begin{abstract}
In this paper, we study a class of flows of closed, star-shaped hypersurfaces in hyperbolic space $\mathbb{H}^{n+1}$ with speed $(\sinh r)^{{\alpha}/{\beta}} \sigma_{k}^{{1}/{\beta}}$, where $\sigma_{k}$ is the $k$-th elementary symmetric polynomial of the principal curvatures, $\alpha$, $ \beta $ are positive constants and $r$ is the distance from points on the hypersurface to the origin. We obtain convergence results under some assumptions of $k$, $\alpha$ and $ \beta $. When $k = 1 , \alpha > 1 + \beta$, and the initial hypersurface is mean convex, we prove that the mean convex solution to the flow for $ k=1 $ exists for all time and converges smoothly to a sphere. When $1\leq k \leq n, \alpha > k+\beta$, and the initial hypersurface is uniformly convex, we prove that the uniformly convex solution to the flow exists for all time and converges smoothly to a sphere. In particular, we generalize Li-Sheng-Wang's results in \cite{L-S-W-2020} from Euclidean space to hyperbolic space.
\end{abstract}

\maketitle
\tableofcontents

\section{Introduction}\label{sec:1}

Let $\mathcal{M}_{0}$ be a smooth ,closed, and star-shaped hypersurface in $\HH^{n+1}$ which encloses the origin, $n \geq 2$. In this paper, we study the following geometric flow,
\begin{equation}\label{hflow}
	\left\{\begin{aligned}
		\frac{\partial}{\partial t} X(x, t) &=-{\left(\phi(r)\right)}^{\frac{\alpha}{\beta}} \sigma_{k}^{\frac{1}{\beta}}(\kappa) \nu(x, t)+\gamma V(x, t), \\
		X(x, 0) &=X_{0}(x) \in \MM_{0},
	\end{aligned}\right.
\end{equation}
where $ \alpha,\beta $ are positive constants, $ \gamma = \tbinom{n}{k} $, $\sigma_{k}$ is the $k$-curvature, given by
$$
\sigma_{k}(\cdot, t)=\sum_{i_{1}<\cdots<i_{k}} \kappa_{i_{1}} \cdots \kappa_{i_{k}}.
$$
and $\kappa_{i}=\kappa_{i}(\cdot, t)$ are the principal curvatures of the hypersurface $\mathcal{M}_{t}$, parametrized by $X(\cdot, t): \mathbb{S}^{n} \rightarrow \HH^{n+1}$, and $\nu(\cdot, t)$ is the unit outer normal of $\mathcal{M}_{t}$ at $X(\cdot, t)$. We denote by $r$ the distance from the point $X(x, t)$ to the origin, $\phi(r)=\sinh r$, and $ V(x,t) = \phi(r)\frac{\p}{\p r} $ is a conformal Killing vector field on $ \HH^{n+1} $.  We shall call $r$ the radial function of $\mathcal{M}_{t}$ in this paper.

The same kind of flow as \eqref{hflow} has been studied by Li, Sheng and Wang in the Euclidean space \cite{L-S-W-2020}. The purpose of this paper is to generalize their results from the Euclidean space to the hyperbolic space. The motivation of studying the flow \eqref{hflow} is its relationship with the prescribed curvature measure problem in hyperbolic space which has been studied recently by Fengrui Yang \cite{Yang20}.

\begin{defn}$\ $

\begin{enumerate}
  \item We say that a smooth closed hypersurface $\mathcal{M}$ is $ k $-convex, $k=1,\cdots,n$, if at every point its principal curvatures satisfy $\kappa=(\kappa_{1},\cdots,\kappa_{n}) \in \Gamma_{k}^{+}$, where
$$
	\Gamma_{k}^{+}:=
	\left\{ (x_1,\cdots,x_n)\in \RR^{n}:\sigma_{i}(x_1,\cdots,x_n)>0, \text{  for all  } 1\leq i\leq k \right\},
$$
and $ \sigma_{k} $ is the $ k $-th elementary symmetric polynomial. For $k=1$, $1$-convex is also called mean convex.
    \item We say that a smooth closed hypersurface $\mathcal{M}$ is uniformly convex if its principal curvatures satisfy $\kappa_i>0$ for all $i=1,\cdots,n$.
    \item A hypersurface $\mathcal{M}$ in the hyperbolic space $\mathbb{H}^{n+1}$ is called star-shaped (with respect to the origin), if its support function $u=\langle V, \nu\rangle$ is positive everywhere, where $V=\phi(r)\frac{\p}{\p r}$.
\end{enumerate}
\end{defn}

\subsection{Main results}
We now state our main results.

\begin{thm}\label{k-convex.converge}
	Let $\MM_{0}$ be a smooth, closed, mean convex and star-shaped hypersurface in $\mathbb{H}^{n+1}$ enclosing the origin. If $ 0<\beta\leq 1 $, $\alpha > \beta+1$, the flow \eqref{hflow} with $ k=1 $ has a unique smooth, mean convex and star-shaped solution $\MM_{t}$ for all time $t>0$. The flow converges exponentially to a sphere centered at the origin in the $C^{\infty}$ topology.
\end{thm}

Moreover, when $ \MM_{0} $ is uniformly convex, we can prove the solution to \eqref{hflow} remains uniformly convex for all $ k=1,\cdots,n$.
\begin{thm}\label{u.convex.converge}
	Let $\MM_{0}$ be a smooth, closed, uniformly convex star-shaped hypersurface in $\mathbb{H}^{n+1}$ enclosing the origin. If $ 0<\beta\leq 1 $, $\alpha > \beta+k$, the flow \eqref{hflow} has a unique smooth and  uniformly convex solution $\MM_{t}$ for all time $t>0$. The flow converges exponentially to a sphere centered at the origin in the $C^{\infty}$ topology.
\end{thm}

In order to prove our main results in Theorem \ref{k-convex.converge} and Theorem \ref{u.convex.converge}, we shall establish the a priori estimates for the flow \eqref{hflow}, and show that if $X(\cdot, t)$ solves \eqref{hflow}, then the radial function $r$ converges exponentially to a constant as $t \rightarrow \infty$.

\subsection{Organization of the paper}
This paper is organized as follows. In section \ref{sec:2} we collect some properties of star-shaped hypersurfaces, and show that the flow \eqref{hflow} can be reduced to a parabolic equation of the radial function. We will also derive the evolution equations for various geometric quantities in section \ref{sec:2}. In section \ref{sec:3} we establish the needed a priori estimates of the radial function $ r $ and its gradient $ \bar{\nabla}r $ of the solution $ \MM_{t} $. In section \ref{sec:4} we establish the needed a priori estimates of the $ k $-curvature $ \sigma_{k} $ of the solution $ \MM_{t} $. In section \ref{sec:5} and section \ref{sec:6} we respectively prove the needed a priori estimates of the principal curvatures for mean convex and uniformly convex solutions. In section \ref{sec:7} we prove the main results of asymptotic convergence.

\begin{ack}
Special thanks to the author's supervisor Prof. Yong Wei. He suggested the author to study the flow in hyperbolic space and helped a lot in improving the rigorousness of the paper. The research was supported by ``Undergraduate Research Program'' and the Research grant KY0010000052 from University of Science and Technology of China.
\end{ack}

\section{Preliminaries}\label{sec:2}
In this paper we view $ \HH^{n+1} $ as the warped product manifold with a coordinate chart homeomorphic to $\RR^{n+1} $ with Riemannian metric $ g^{\mathbb{H}} = dr\otimes dr + (\phi(r))^2 \bar{g} $, where $ \bar{g} $ is the standard metric on $ \mathbb{S}^{n} $ and $\phi(r)=\sinh r$.
\subsection{The vector field V}
We denote the covariant derivative with respect to $ g^{\mathbb{H}} $ in  $ \HH^{n+1} $ by $ D $. We have the following basic properties for the vector field $V=\phi(r)\partial_r$.
\begin{lem}[see \cite{Guan-2015}]\label{propertyV}
	For any tangent vector field $ X $ on $ \HH^{n+1} $, we have
	$$ D_XV = \phi'(r)X. $$
\end{lem}
\begin{cor}\label{s2.cor1}
	For any tangent vector field $ X $ on $ \HH^{n+1} $, we have
	$$D_{X}r = \frac{\left\langle X,V \right\rangle}{\phi(r)}. $$
	where $ \left\langle\cdot,\cdot\right\rangle $ means the inner product of tangent vectors induced by $ g^{\HH} $.
\end{cor}
\begin{proof}
	Since $ V = \phi(r)\p_r $, $ \left\langle V,V\right\rangle = (\phi(r))^2 $. 	Thus
	$$
	D_{X}\left\langle V,V\right\rangle = D_{X}\left((\phi(r))^2\right) =
	2\phi(r)\phi'(r)D_{X}r.
	$$
	Meanwhile, by \lemref{propertyV}, for the left hand side we have
	$$
	D_{X}\left\langle V,V\right\rangle =
	2\left\langle D_{X} V,V\right\rangle =
	2\left\langle \phi'(r)X,V\right\rangle =
	2\phi'(r)\left\langle X,V\right\rangle.
	$$
	Hence
	$ D_{X}r = \frac{\left\langle X,V \right\rangle}{\phi(r)} $.
\end{proof}

\subsection{Elementary symmetric polynomials}\hfil\\
For $ n $ variables $ x_1,\cdots,x_n $, their $ k $ th elementary symmetric polynomial is defined by
$$
\sigma_{k}(x_1,\cdots,x_n) = \sum_{1\leq i_1<\cdots<i_k \leq n}x_{i_1}\cdots x_{i_n},
$$
%

\begin{lem}[Newton-Maclaurin's inequalities]\label{Maclaurin}
	When $ (x_1,\cdots,x_n) \in \Gamma_{m-1}^{+} $, for $ l, m $ such that $ 1\leq l < m \leq n $, we have
	\begin{equation}
		\sigma_m (x_1,\cdots,x_n)\leq \tbinom{n}{m}\left(\frac{\sigma_l (x_1,\cdots,x_n)}{\tbinom{n}{l}}\right)^{{m}/{l}}.
	\end{equation}
Equality holds if and only if $(x_1,\cdots,x_n)=c(1,\cdots,1)$ for some constant $c>0$.
\end{lem}

\begin{lem}[see \cite{Guan-2014}]\label{dsigma}
	For a real symmetric matrix $ A=(a_{ij})_{1\leq i,j\leq n} $, we write $\sigma_{k}(A) = \sigma_{k}\left(\lambda_1,\ldots,\lambda_n\right)$, where $ \left(\lambda_1,\ldots,\lambda_n\right) $ are the eigenvalues of A. We write
	$ \dot{\sigma_{k}}^{ij} = \frac{\p \sigma_{k}}{\p a_{ij}} $.
 When the eigenvalues of A belong to $\Gamma_k^+$,  $ k=1,\cdots, n $ , we have that $ (\dot{\sigma_{k}}^{ij} ) $ is positive definite.
\end{lem}

In the most case of the following paper, by using $ \sigma_{k} $, we mean the $ k $ th elementary symmetric polynomial of the principal curvatures, that is, locally, eigenvalues of the matrix $ \left(h_{i}^{j}\right) $.
For consideration of the consistency of superscripts and subscripts, by using the notation $ \dot{\sigma_{k}}^{ij} $, we mean
$ \dot{\sigma_{k}}^{ij} = g^{jl}\frac{\p \sigma_{k}}{\p h_{i}^{l}} $, and similarly we can define second derivatives by $ \ddot{\sigma_{k}}^{pq,rs} = g^{ql}g^{sm}\frac{\p^2 \sigma_{k}}{\p h_{p}^{l}\p h_{r}^{m}} $.

For general functions on the eigenvalues of a real symmetric matrix, such as $ F=(\sigma_{k})^{\frac{1}{\beta}} $, we also write
$$ \dot{F}^{ij} = g^{jl}\frac{\p F}{\p h_{i}^{l}}, \text{ and } \ddot{F}^{pq,rs} = g^{ql}g^{sm}\frac{\p^2 F}{\p h_{p}^{l}\p h_{r}^{m}}. $$

We recall the following result:
\begin{lem}[see \cite{Guan-2014}]\label{trace}
	\begin{align*}
	\dot{\sigma_{k}}^{ij}(h^2)_{ij} = & \sigma_1\sigma_{k}-(k+1)\sigma_{k+1},\\
	\dot{\sigma_{k}}^{ij}h_{ij} = & k\sigma_{k},\\
	\dot{\sigma_{k}}^{ij}g_{ij} = &(n-k+1)\sigma_{k-1},
\end{align*}
where $ (h^2)_{ij} = h^{ik}h_{kj} $.
\end{lem}

For $ F = (\sigma_{k})^{\frac{1}{\beta}} $, we have $\dot{F}^{ij}=\frac{1}{\beta}(\sigma_{k})^{\frac{1}{\beta}-1} \dot{\sigma_{k}}^{ij}$ and
\begin{cor}\label{F.trace}
		\begin{align*}
			\dot{F}^{ij}(h^2)_{ij} = & \frac{1}{\beta}(\sigma_{k})^{\frac{1}{\beta}-1} \left(\sigma_1\sigma_{k}-(k+1)\sigma_{k+1}\right),\\
			\dot{F}^{ij}h_{ij} = & \frac{k}{\beta}(\sigma_{k})^{\frac{1}{\beta}},\\
			\dot{F}^{ij}g_{ij} = & \frac{n-k+1}{\beta}(\sigma_{k})^{\frac{1}{\beta}-1}\sigma_{k-1}.
		\end{align*}
\end{cor}

\subsection{Radial function and radial graph}

Let $e_{1}, \ldots, e_{n}$ be a smooth local orthonormal frame field on $\mathbb{S}^{n}$, and let $\bar{\nabla}$ be the covariant derivative on $\mathbb{S}^{n}$. We denote by $g_{i j}, g^{i j}, \nu, h_{i j}$ the metric on $ \MM_{t} $ induced by the metric of $ \HH^{n+1} $, the inverse of the metric, the unit outer normal and the second fundamental form of $\mathcal{M}_t$, respectively. Then, in terms of $r$, we have
\begin{equation}
	\begin{aligned}\label{radial.quantities}
		u &=\frac{\phi^{2}}{\sqrt{\phi^{2}+|\bar{\nabla} r|^{2}}} \\
		g_{i j} &=\phi^{2} \delta_{i j}+r_{i} r_{j},
		\quad g^{i j}=\frac{1}{\phi^{2}}\left(\delta^{i j}-\frac{r_{i} r_{j}}{\phi^{2}+|\bar{\nabla} r|^{2}}\right) \\
		h_{i j} &=\left(\sqrt{\phi^{2}+|\bar{\nabla} r|^{2}}\right)^{-1}\left(-\phi \bar{\nabla}_{i}\bar{ \nabla}_{j} r+2 \phi^{\prime} r_{i} r_{j}+\phi^{2} \phi^{\prime}\delta_{i j}\right) \\
		h_{j}^{i} &=\frac{1}{\phi^{2} \sqrt{\phi^{2}+|\bar{\nabla} r|^{2}}}\left(\delta^{i k}-\frac{r_{i} r_{k}}{\phi^{2}+|\bar{\nabla} r|^{2}}\right)\left(-\phi \bar{\nabla}_{k} \bar{\nabla}_{j} r+2 \phi^{\prime} r_{k} r_{j}+\phi^{2} \phi^{\prime} \delta_{k j}\right),
	\end{aligned}
\end{equation}
where we denote $ r_i=\bar{\nabla}_i(r) $. These formulae can be found in a number of papers, see, for example \cite{Guan-2015}.

It will be convenient if we introduce a new variable $\varphi = \varphi(r)$ satisfying
\begin{equation}\label{varphi.def}
	\frac{\mathrm{d} \varphi}{\mathrm{d} r}=\frac{1}{\phi(r)}.
\end{equation}
\\
Let $\omega:=\sqrt{1+|\bar{\nabla} \varphi|^{2}}$, one can calculate the unit outward normal $$\nu=\frac{1}{\omega}\left(1,-\frac{r_{1}}{\phi^{2}}, \ldots,-\frac{r_{n}}{\phi^{2}}\right)$$ and the general support function $u=\langle V, \nu\rangle=\frac{\phi}{\omega} $. Then we can conclude how the radial function evolves along the flow.

\begin{lem}
The star-shaped solution of the flow
$$ \frac{\partial}{\partial t} X(x, t) =-{\left({\phi(r)}\right)}^{\frac{\alpha}{\beta}} \sigma_{k}^{\frac{1}{\beta}}(\kappa) \nu(x, t)+\gamma V(x,t)
$$
is equivalent to a solution to the scalar parabolic PDE of the radial function $r$:
\begin{equation}\label{hradial.eq}
	\frac{\partial}{\partial t} r =-{\left({\phi(r)}\right)}^{\frac{\alpha}{\beta}} \sigma_{k}(\kappa)^{\frac{1}{\beta}} \omega+\gamma \phi.
\end{equation}
\end{lem}
\begin{proof}
	Up to a tangential diffeomorphism, the flow \eqref{hflow} is equvalent to the following flow
	$$
	\frac{\partial}{\partial t} X(x, t) =\left(-{\left({\phi(r)}\right)}^{\frac{\alpha}{\beta}} \sigma_{k}(\kappa)^{\frac{1}{\beta}} +\gamma u\right)\nu(x, t).
	$$
	By \cite{Gerhardt-2007}, It is known that if a closed hypersurface which is a radial graph satisfies
	$$
	\partial_{t} X=f \nu,
	$$
	then the evolution of the scalar function $r$ satisfies
	$$
	\partial_{t} r=f \omega.
	$$
	Thus
	$$
	\frac{\partial}{\partial t} r =\left(-{\left({\phi(r)}\right)}^{\frac{\alpha}{\beta}} \sigma_{k}(\kappa)^{\frac{1}{\beta}} +\gamma u\right)\omega
	=-{\left({\phi(r)}\right)}^{\frac{\alpha}{\beta}} \sigma_{k}(\kappa)^{\frac{1}{\beta}} \omega+\gamma \phi.
	$$
\end{proof}

\subsection{Evolution of geometric quantities}

For convenience, we rewrite \eqref{hflow} as
\begin{equation}\label{flow}
	\frac{\partial}{\partial t} X(x, t)=-\Phi \nu(x, t)+\gamma V(x,t).
\end{equation}
where $\gamma= \tbinom{n}{k}$ and
\begin{equation}\label{Phi.def}
	\Phi=
		{\left({\phi(r)}\right)}^{\frac{\alpha}{\beta}} \sigma_{k}(\kappa)^{\frac{1}{\beta}}.
\end{equation}

We now derive some evolution equations along the flow \eqref{flow}. Pick any local coordinate chart $\left\{x_{i}\right\}_{i=1}^{n}$ of the hypersurface $ \MM_{t} $. We denote $ \nabla $ to be the covariant derivative on $ \MM_{t} $ induced by the covariant derivative $ D$ in $ \HH^{n+1} $. We denote $ g_{ij} $ and $ \left\langle\cdot, \cdot \right\rangle $ to be respectively the metric on $ \MM_{t} $ induced by $ g^{\HH} $ and the induced inner product of tangent vectors. We denote $\partial_{i}=\frac{\partial}{\partial x_{i}}, X_{i}=\partial_{i} X$. Recall the following identities in hyperbolic space, see \cite{AndrewsChen-2017}, \cite{Guan-2021}.
\begin{align}
		D_{X_i}X_j&=\Gamma_{ij}^kX_k-h_{i j} \nu,&\text {  (Gauss formula)}\label{basic.identities1}\\
		\nu_{i}&=h_{i j} g^{j l} X_{l}, &\text{  (Weingarten formula)}\label{basic.identities2}\\
		h_{i j, l}&=h_{i l, j},  &\text {  (Codazzi equation)}\label{basic.identities3} \\
		R_{i j r s}&=h_{i r} h_{j s}-h_{i s} h_{j r}-g_{ir}g_{js}+g_{is}g_{jr},  &\text {  (Gauss equation)}\label{basic.identities4}
\end{align}
where $\Gamma_{i j}^{k}$ is the Christoffel symbol of the metric of $\mathcal{M}_{t} $. Combining the Gauss and Codazzi equations gives the following generalization of Simons' identity, we have
\begin{equation}\label{basic.identities5}
	\nabla_{(i} \nabla_{j)} h_{k l}=\nabla_{(k} \nabla_{l)} h_{i j}+h_{i j} (h^2)_{kl}-h_{k l} (h^2)_{ij}-g_{i j} h_{k l}+g_{k l} h_{i j}
\end{equation}
where $ (h^2)_{ij} = h_{i}^{k}h_{kj} $ and $ A_{(ab)} $ means the symmetrization of the tensor $ A_{ab} $, that is, $ A_{(ab)} = \frac{1}{2}\left(A_{ab}+A_{ba}\right) $.
\begin{lem}\label{evolve.quantities}
	Along the flow \eqref{flow}, we have the following evolution equations
	\begin{equation}
		\frac{\partial}{\partial t} g_{i j} =-2 \Phi h_{i j}+2 \gamma\phi' g_{i j}, \quad \frac{\partial}{\partial t} \nu=\nabla \Phi.
	\end{equation}
	\begin{equation}
		\frac{\partial}{\partial t} u = -\phi'\Phi+\phi'\gamma u + \langle V,\nabla\Phi\rangle.
	\end{equation}
	\begin{equation}
		\frac{\partial}{\partial t} h_{i j} = \nabla_{i} \nabla_{j} \Phi-\Phi h_{i}^{l} h_{l j}+\gamma\phi' h_{i j} + (\gamma u -\Phi)g_{ij}.
	\end{equation}
	Here $ \nabla $ denotes the covariant derivative on $\mathcal{M}_{t}$.
\end{lem}

\begin{proof}By direct calculations using Lemma \ref{propertyV} and the equation \eqref{basic.identities2}, we have
	\begin{align*}
		\frac{\partial}{\partial t} g_{i j} &=\partial_{t}\left\langle\partial_{i} X, \partial_{j} X\right\rangle \\
		&=\left\langle\nabla_{i}(-\Phi \nu+\gamma V), \partial_{j} X\right\rangle+\left\langle\partial_{i} X, \nabla_{j}(-\Phi \nu+\gamma V)\right\rangle \\
		&=-\Phi\left(\left\langle\partial_{i} \nu, \partial_{j} X\right\rangle+\left\langle\partial_{i} X, \partial_{j} \nu\right\rangle\right)+2 \gamma\phi' g_{i j} \\
		&=-2 \Phi h_{i j}+2 \gamma\phi' g_{i j}.
	\end{align*}
	Since $\partial_t\nu$ is tangential, we have
	\begin{align*}
		&\frac{\partial}{\partial t} \nu =\left\langle\p_{t} \nu, \nabla_{j} X\right\rangle g^{i l} \partial_{l} X \\ =&-\left\langle\nu, \nabla_{j}(-\Phi \nu+\gamma V)\right\rangle g^{j l} \partial_{l} X =\nabla_{j} \Phi g^{j l} \partial_{l} X =\nabla \Phi .
	\end{align*}
	
	Now we calculate the evolution of the support function $ u $
	\begin{align*}
		&\frac{\partial}{\partial t} u
		=\partial_{t}\langle V, \nu\rangle
		=\langle \p_{t}V, \nu\rangle + \langle V, \p_{t} \nu\rangle
		=\langle \phi'\p_{t}X,\nu\rangle + \langle V, \p_{t} \nu\rangle\\
		=&\langle\phi'(-\Phi \nu+\gamma V), \nu\rangle+\langle V, \nabla \Phi\rangle =-\phi'\Phi+\phi'\gamma u + \langle V,\nabla\Phi\rangle.
	\end{align*}
	
	Now we calculate the evolution of $h_{i j}$. Note that $ \HH^{n+1} $ is of constant sectional curvature of $ K = -1 $, thus
	\begin{align*}
		&\frac{\partial}{\partial t} h_{i j}
		=-{\partial_t}\left\langle D_{X_{i}} X_{j} , \nu\right\rangle\\
		=&-\left\langle D_{X_{t}} D_{X_{i}} X_{j}, \nu \right\rangle-\left\langle D_{X_{i}} X_{j}, \partial_{t} \nu\right\rangle\\
		=&-\left\langle D_{X_{i}} D_{X_{t}} X_{j}+\bar{R}\left(X_{t}, X_{i}\right) X_{j}, \nu\right\rangle-\left\langle D_{X_{i}} X_{j}, \nabla \Phi\right\rangle\\
		=&-\left\langle D_{X_{i}} D_{X_{j}} X_{t},\nu\right\rangle-\left\langle\bar{R}\left(X_{t}, X_{i}\right) X_{j}, \nu\right)-\left\langle D_{X_{i}} X_{j}, g^{k l} \partial_{l} \Phi X_{k}\right\rangle\\
		=&-\left\langle D_{X_{i}} D_{X_{j}}(-\Phi \nu+\gamma V), \nu\right\rangle+
		\left\langle X_{t}, \nu\right\rangle
		\left\langle X_{i}, X_{j}\right\rangle
		-\Gamma_{i j}^{k} \partial_{k} \Phi\\
		=&-\left\langle D_{X_{i}}\left(-\partial_{j} \Phi \nu-\Phi h_{j}^{l} X_{l}+\gamma \phi' X_{j}\right), \nu\right\rangle+g_{i j}(\gamma u-\Phi)-\Gamma_{i j}^{k} \partial_{k} \Phi\\
		=&\partial_{i} \partial_{j} \Phi-\Gamma_{i j}^{k} \partial_{k} \Phi
		-\Phi h_{j}^{l} h_{l_{i}}+\gamma{\phi}' h_{i j}+g_{i j}(\gamma u-\Phi)\\
		=&\nabla_{i}\nabla_{j}\Phi
		-\Phi h_{j}^{l} h_{l_{i}}+\gamma{\phi}' h_{i j}+g_{i j}(\gamma u-\Phi),
	\end{align*}
where $\bar{R}$ denotes the curvature tensor of the metric on $\mathbb{H}^{n+1}$.
\end{proof}

\begin{cor}\label{evolve.quantities1}
	Along the flow \eqref{flow}, we have the following evolution equations
	\begin{equation}
		\frac{\partial}{\partial t} g^{ij} = 2 \Phi h^{i j}-2 \gamma\phi' g^{i j}.
	\end{equation}
	\begin{equation}
		\frac{\partial}{\partial t} h_{i}^{j} = \nabla_{i}\nabla^{j}\Phi+\Phi h_{i}^{k} h_{k}^{j} - \gamma\phi'h_{i}^{j} +(\gamma u -\Phi)\delta_{i}^{j}.
	\end{equation}
\end{cor}

\begin{proof}
	$$
	\frac{\partial}{\partial t} g^{i j}=-g^{i l}\left(\partial_{t} g_{l m}\right) g^{m j}=2 \Phi h^{i j}-2 \gamma\phi' g^{i j}.
	$$
	$$
	\begin{aligned}
		\frac{\partial}{\partial t} h_{i}^{j}
		&=\partial_{t} h_{i l} g^{l j}+h_{i l} \partial_{t} g^{l j} \\
		&=\left(\nabla_{i} \nabla_{l} \Phi-\Phi h_{i}^{k} h_{k l}+\gamma\phi' h_{i l}+(\gamma u - \Phi)g_{il}\right)g^{lj} + h_{i l}\left(2 \Phi h^{l j}-2 \gamma\phi' g^{l j}\right) \\
		&=\nabla_{i}\nabla^{j}\Phi+\Phi h_{i}^{k} h_{k}^{j} - \gamma\phi'h_{i}^{j} +(\gamma u -\Phi)\delta_{i}^{j}.
	\end{aligned}
	$$
\end{proof}
Now we calculate $ \nabla_{i}\nabla_{j}\Phi $ more precisely. Recall that $\Phi=\phi^{\frac{\alpha}{\beta}}F$, where $F=\sigma_k^{\frac{1}{\beta}}$.
\begin{lem}
		\begin{align}
			\nabla_{(i}\nabla_{j)}\Phi
			\nonumber
			=&(\phi(r))^{\frac{\alpha}{\beta}}\dot{F}^{pq}\nabla_{p}\nabla_{q}h_{ij}
			+ (\phi(r))^{\frac{\alpha}{\beta}}\ddot{F}^{pq,rs}h_{rs,i}h_{pq,j} \\
			\nonumber
			&+\left(\frac{\alpha}{\beta}\right)^2\frac{(\phi'(r))^2}{(\phi(r))^2}(\nabla_{i}r)(\nabla_{j}r)
			\Phi -
			\frac{\alpha}{\beta}\frac{1}{(\phi(r))^2}(\nabla_{i}r)(\nabla_{j}r)\Phi -
			\frac{\alpha}{\beta}\frac{\phi'(r)}{(\phi(r))^2}uh_{ij}\Phi \\
			\nonumber
			&+\frac{\alpha}{\beta}\frac{(\phi'(r))^2}{(\phi(r))^2}
			\left(
			g_{ij}-(\nabla_{i}r)(\nabla_{j}r)
			\right)\Phi
			+2\frac{\alpha}{\beta}\phi'(r) (\phi(r))^{\frac{\alpha}{\beta}-1}
			(\nabla_{(i}r)(\nabla_{j)}F)
			\\
			\nonumber
			&+
			\frac{1}{\beta}(\phi(r))^{\frac{\alpha}{\beta}}(\sigma_{k})^{\frac{1}{\beta}-1}
			\left(\sigma_1\sigma_{k}-(k+1)\sigma_{k+1}\right)h_{ij}
			-\frac{k}{\beta}\Phi\left(h^2\right)_{ij} \\
			\label{bigresult}
			&-\frac{k}{\beta}\Phi g_{ij} +
			\frac{n-k+1}{\beta}(\sigma_{k})^{\frac{1}{\beta}-1}(\phi(r))^{\frac{\alpha}{\beta}}\sigma_{k-1}h_{ij}.
		\end{align}
\end{lem}
\begin{proof}
	Firstly
		\begin{align*}
			\nabla_{i}\Phi
			= &\nabla_{i}\left( (\phi(r))^{\frac{\alpha}{\beta}} F \right)
			= \frac{\alpha}{\beta}(\phi(r))^{\frac{\alpha}{\beta}-1}\phi'(r)(\nabla_{i}r)
			F +
			(\phi(r))^{\frac{\alpha}{\beta}} \nabla_{i}F\\
			=&\frac{\alpha}{\beta}\frac{\phi'(r)}{\phi(r)}\Phi\nabla_{i}r +
			(\phi(r))^{\frac{\alpha}{\beta}} \nabla_{i}F.
		\end{align*}
	Therefore,
	\begin{equation}
		\begin{aligned}\label{temp3}
			\nabla_{i}\nabla_{j}\Phi
			=& \nabla_{i}\left( \frac{\alpha}{\beta}\frac{\phi'(r)}{\phi(r)}\Phi\nabla_{j}r +
			(\phi(r))^{\frac{\alpha}{\beta}} \nabla_{j}F \right) \\
			=& -\frac{\alpha}{\beta}\frac{1}{(\phi(r))^2}(\nabla_{i}r)(\nabla_{j}r)\Phi +
			\frac{\alpha}{\beta}\frac{\phi'(r)}{\phi(r)}(\nabla_{i}\nabla_{j}r)\Phi \\
&+\frac{\alpha}{\beta}\frac{\phi'(r)}{\phi(r)}(\phi(r))^{\frac{\alpha}{\beta}}\nabla_{i}r \nabla_{j}F
			+			(\phi(r))^{\frac{\alpha}{\beta}}\nabla_{i}\nabla_{j}F \\
			&+ \frac{\alpha}{\beta}\frac{\phi'(r)}{\phi(r)}\nabla_{j}r
			\left(
			\frac{\alpha}{\beta}\frac{\phi'(r)}{\phi(r)}\Phi\nabla_{i}r +
			(\phi(r))^{\frac{\alpha}{\beta}} \nabla_{i}F
			\right).
		\end{aligned}
	\end{equation}
	Note that by \corref{s2.cor1},
		\begin{align}
			\nonumber
			&\nabla_{i}\nabla_{j}r
			=\nabla_{i}\left( \frac{\langle X_j,V\rangle}{\phi(r)} \right)
			=\frac{\langle \nabla_{i} X_j,V\rangle}{\phi(r)} +
			\frac{\langle X_j,\nabla_{i} V\rangle}{\phi(r)} -
			\frac{\phi'(r)}{(\phi(r))^2}(\nabla_{i}r)\langle X_j,V\rangle \\
			\nonumber
			=&\frac{\langle -h_{ij}\nu,V\rangle}{\phi(r)} +
			\frac{\langle X_j,\phi'(r) X_i\rangle}{\phi(r)} -
			\frac{\phi'(r)}{\phi(r)}(\nabla_{i}r)(\nabla_{j}r) \\
			\label{temp4}
			=&-\frac{u}{\phi(r)}h_{ij} + \frac{\phi'(r)}{\phi(r)}g_{ij} - \frac{\phi'(r)}{\phi(r)}(\nabla_{i}r)(\nabla_{j}r).
		\end{align}
	We also have
	\begin{equation}\label{temp5}
			\nabla_{i}\nabla_{j}F
			=\nabla_{i}\left( \dot{F}^{pq}\nabla_{j}h_{pq} \right)
			=\ddot{F}^{pq,rs}\nabla_{i}h_{rs}\nabla_{j}h_{pq} +
			\dot{F}^{pq}\nabla_{i}\nabla_{j}h_{pq}
	\end{equation}
	Plugging \eqref{temp4} and \eqref{temp5} into \eqref{temp3} and symmetrize the tensor $ \nabla_{i}\nabla_{j}\Phi $, we get
	\begin{align*}
		\nabla_{(i}\nabla_{j)}\Phi
		=&\left(\frac{\alpha}{\beta}\right)^2
		\frac{(\phi'(r))^2}{(\phi(r))^2}(\nabla_{i}r)(\nabla_{j}r)
		\Phi -
		\frac{\alpha}{\beta}\frac{1}{(\phi(r))^2}(\nabla_{i}r)(\nabla_{j}r)\Phi \\
		&+\frac{\alpha}{\beta}\frac{\phi'(r)}{\phi(r)}
		\left(
		-\frac{u}{\phi(r)}h_{ij} + \frac{\phi'(r)}{\phi(r)}g_{ij} - \frac{\phi'(r)}{\phi(r)}(\nabla_{i}r)(\nabla_{j}r)
		\right)\Phi \\
		&+
		2\frac{\alpha}{\beta}\phi'(r) (\phi(r))^{\frac{\alpha}{\beta}-1}
		(\nabla_{(i}r)(\nabla_{j)}F)
		+(\phi(r))^{\frac{\alpha}{\beta}}\nabla_{(i}\nabla_{j)}F \\
		=&\left(\frac{\alpha}{\beta}\right)^2
		\frac{(\phi'(r))^2}{(\phi(r))^2}(\nabla_{i}r)(\nabla_{j}r)
		\Phi -
		\frac{\alpha}{\beta}\frac{1}{(\phi(r))^2}(\nabla_{i}r)(\nabla_{j}r)\Phi  \\
		&-\frac{\alpha}{\beta}\frac{\phi'(r)}{(\phi(r))^2}uh_{ij}\Phi + \frac{\alpha}{\beta}\frac{(\phi'(r))^2}{(\phi(r))^2}
		\left(
		g_{ij}-(\nabla_{i}r)(\nabla_{j}r)
		\right)\Phi \\
		&+2\frac{\alpha}{\beta}\phi'(r) (\phi(r))^{\frac{\alpha}{\beta}-1}
		(\nabla_{(i}r)(\nabla_{j)}F) \\
		&+(\phi(r))^{\frac{\alpha}{\beta}}
		\left(
		\ddot{F}^{pq,rs}h_{rs,i}h_{pq,j}+
		{F}^{pq}\nabla_{(i}\nabla_{j)}h_{pq}
		\right).
	\end{align*}
	By using generalisation of Simons' identity in \eqref{basic.identities5}, we have that
	$$
	\nabla_{(i} \nabla_{j)} h_{k l}=\nabla_{(k} \nabla_{l)} h_{i j}+h_{i j} (h^2)_{kl}-h_{k l} (h^2)_{ij}-g_{i j} h_{k l}+g_{k l} h_{i j}.
	$$
	Thus
	\begin{align*}
		\nabla_{(i}\nabla_{j)}\Phi
		=&\left(\frac{\alpha}{\beta}\right)^2\frac{(\phi'(r))^2}{(\phi(r))^2}(\nabla_{i}r)(\nabla_{j}r)
		\Phi -
		\frac{\alpha}{\beta}\frac{1}{(\phi(r))^2}(\nabla_{i}r)(\nabla_{j}r)\Phi -
		\frac{\alpha}{\beta}\frac{\phi'(r)}{(\phi(r))^2}uh_{ij}\Phi \\ &+\frac{\alpha}{\beta}\frac{(\phi'(r))^2}{(\phi(r))^2}
		\left(
		g_{ij}-(\nabla_{i}r)(\nabla_{j}r)
		\right)\Phi + (\phi(r))^{\frac{\alpha}{\beta}}\ddot{F}^{pq,rs}h_{rs,i}h_{pq,j} \\
		&+
		2\frac{\alpha}{\beta}\phi'(r) (\phi(r))^{\frac{\alpha}{\beta}-1}
		(\nabla_{(i}r)(\nabla_{j)}F) \\
		&+(\phi(r))^{\frac{\alpha}{\beta}}\dot{F}^{pq}
		\left(
		\nabla_{(p} \nabla_{q)} h_{i j}+h_{i j} (h^2)_{pq}-h_{pq} (h^2)_{ij}-
		g_{i j} h_{pq}+g_{pq} h_{i j}
		\right).
	\end{align*}
	Plugging \corref{F.trace} into the last line of the above equation, we obtain \eqref{bigresult}.
\end{proof}

\section{$ C^0 $ and $ C^1 $ estimates}\label{sec:3}

In this section, we establish the needed a priori estimates of the radial function $ r $ and its gradient $ \bar{\nabla}r $ of the solution $ \MM_{t} $.

\begin{lem}\label{C0estimate}
	Let $r(\cdot, t)$ be a smooth solution to \eqref{hradial.eq} on $\mathbb{S}^{n} \times[0, T)$.
	If $\alpha > k+\beta$, there is a positive constant $C$ depending only on k, $\alpha$, $\max _{\mathbb{S}^{n}} r(\cdot, 0)$ and $\min _{\mathbb{S}^{n}} r(\cdot, 0)$ such that
	$$
	1 / C \leq r(\cdot, t) \leq C \quad \forall t \in[0, T).
	$$
\end{lem}

\begin{proof}
	By \eqref{hradial.eq}, we conclude that
	$$
	\frac{\p r}{\p t} = -\Phi\omega+\gamma \phi.
	$$
	Here by \eqref{Phi.def},
	$$
	\Phi=
		{\left({\phi(r)}\right)}^{\frac{\alpha}{\beta}} \sigma_{k}(\kappa)^{\frac{1}{\beta}}.
	$$
	For each time $ t $, if $ r(x,t) $ attains its spatial minimum at the point $ (x_t,t) $, then we have
		\begin{align*}
			&\bar{\nabla}_{i}r = 0,\text{  for all  }1\leq i \leq n,\\
			&\bar{\nabla}_{i}\bar{\nabla}_{j}r \geq 0,\text{  as a matrix}.
		\end{align*}
	Hence $ \omega = \sqrt{1+|\bar{\nabla} \varphi|^2} = \sqrt{1+\frac{|\bar{\nabla} r|^2}{\phi^2}}  = 1$ at $ (x_t,t) $. Therefore we have $ u = \frac{\phi}{\omega} = \phi $ and
	$$ \frac{\p r}{\p t} = -\Phi+\gamma\phi\quad \text{  at  }(x_t,t). $$
	Moreover, by the last equation of \eqref{radial.quantities}, at $ (x_t,t) $,
$$ h_{i}^{j} = h_{ik}g^{kj} \leq \phi\phi'\delta_{i k}\cdot\frac{1}{\phi^2}\delta^{k j} = \frac{\phi'}{\phi}\delta_{i}^{j}  $$
	and $$ \sigma_{k}(\kappa) = \sigma_{k}(h_{i}^{j}) \leq \sigma_{k}(\frac{\phi'}{\phi}\delta_{i}^{j}) = \frac{{\left(\phi'\right)}^k}{\phi^k}\tbinom{n}{k} = \frac{{\left(\phi'\right)}^k}{\phi^k}\gamma. $$
	
	We now prove the existence of uniform positive lower bound of $ r $. The proof of the existence of uniform positive upper bounds is similar.

	Since $ \phi(r) = \sinh(r) $, $ \Phi = {\left({\sinh(r)}\right)}^{\frac{\alpha}{\beta}} \sigma_{k}(\kappa)^{\frac{1}{\beta}}\leq {\left(\sinh(r)\right)}^{\frac{\alpha-k}{\beta}}{\left(\cosh(r)\right)}^{\frac{k}{\beta}}\gamma $, we have
	\begin{align*}
		\frac{\p r}{\p t} &
		= -\Phi+\gamma\phi
		\geq -\gamma{\left(\sinh(r)\right)}^{\frac{\alpha-k}{\beta}}{\left(\cosh(r)\right)}^{\frac{k}{\beta}}\gamma+
		\gamma \sinh(r) \\
		&= -\gamma\sinh(r) \left( {\left(\sinh(r)\right)}^{\frac{\alpha-k-\beta}{\beta}} {\left(\cosh(r)\right)}^{\frac{k}{\beta}}-1 \right).
	\end{align*}

	By $ \alpha > k+\beta $, we know that
	$ {\left(\sinh(r)\right)}^{\frac{\alpha-k-\beta}{\beta}} {\left(\cosh(r)\right)}^{\frac{k}{\beta}} $ is a monotonically increasing function as $ r>0 $.
	This function takes value $ 0 $ when $ r=0 $, and tends to $ +\infty $ and $ r $ tends to $ +\infty $. 	So there exists a unique $ \hat{r} = \hat{r}(k,\alpha,\beta)\in(0,+\infty) $,
	such that $ {\left(\sinh(\hat{r})\right)}^{\frac{\alpha-k-\beta}{\beta}} {\left(\cosh(\hat{r})\right)}^{\frac{k}{\beta}} = 1 $.
	
	Hence if $ r(x_t,t) = r_{\min} (t) \leq \hat{r} $, we have
	$ {\left(\sinh(r)\right)}^{\frac{\alpha-k-\beta}{\beta}} {\left(\cosh(r)\right)}^{\frac{k}{\beta}} \leq 1 $. It follows that
	\begin{equation*}
  	\frac{d}{dt}r_{\min}(t) = \frac{\p r}{\p t}(x_t,t)\geq 0.
	\end{equation*}
	This implies that
	$ r_{\min}(t) \geq \min \{ \min_{\mathbb{S}^{n}}r(\cdot,0), \hat{r}(k,\alpha) \} $.
	That is,
	$$ r(x,t) \geq \min \{ \min_{\mathbb{S}^{n}}r(\cdot,0), \hat{r}(k,\alpha,\beta) \}. $$
	By similar argument, we can prove that
	$$ r(x,t) \leq \max \{ \max_{\mathbb{S}^{n}}r(\cdot,0), \hat{r}(k,\alpha,\beta) \}. $$
\end{proof}
\begin{lem}\label{C1estimate}
	Let $r(\cdot, t)$ be a  smooth $k$-convex solution to \eqref{hradial.eq} on $\mathbb{S}^{n} \times[0, T)$. If $\alpha > k+\beta$, we have
	$$
	|\bar{\nabla} r| \leq C.
	$$
	where $ \bar{\nabla} $ is the covariant derivative on $ \SSS^n $. $C$ is positive constant which only depends on $\mathcal{M}_{0}$.
\end{lem}
\begin{proof}By \eqref{hradial.eq}, we have
	$$
		\p_{t}r = -{\left(\sinh(r)\right)}^{\frac{\alpha}{\beta}} {\left(\sigma_{k}\right)}^{\frac{1}{\beta}} \omega+\gamma \sinh(r).
	$$
	As in \eqref{varphi.def}, take a new function $\varphi(r) = \log\left(1-\frac{2}{e^r+1}\right) $, then $ \varphi $ is monotonically increasing and $ \varphi(r)\in(-\infty,0) $, for all $ r>0 $, and
	it satisfies $$ \frac{d \varphi}{d r} = \frac{1}{\sinh(r)}. $$
	Then by $ \omega=\sqrt{1+|\bar{\nabla} \varphi|^{2}} $,
	and $ r = \log\left(\frac{2}{1-e^{\varphi}}-1\right) $, we have
	$$ \sinh(r) = \frac{1}{\sinh(-\varphi)}  \text{ and } \cosh(r) = \frac{\cosh(-\varphi)}{\sinh(-\varphi)}, $$ So
	$$
		\p_{t}\varphi = -(\sinh(-\varphi))^{-\frac{\alpha}{\beta}+1}\sqrt{1+|\bar{\nabla} \varphi|^2}\sigma_{k}^{\frac{1}{\beta}} + \gamma .
	$$
	
	For a fixed time $t$, we assume that $\theta_{t} \in \mathbb{S}^{n}$ is the spatial maximum point of $\frac{1}{2}|\bar{\nabla} \varphi|^{2}$. At $\left(\theta_{t}, t\right)$, we have
		\begin{align}\label{s3.C1-1}
			\p_{t}\left(\frac{1}{2} |\bar{\nabla} \varphi|^{2} \right)
			= &\varphi_{i}(\p_{t}\varphi)_{i}
			= \varphi_{i}\left[ -(\sinh(-\varphi))^{-\frac{\alpha}{\beta}+1}\sqrt{1+|\bar{\nabla} \varphi|^2}\sigma_{k}^{\frac{1}{\beta}} + \gamma \right]_{i}\nonumber\\
			=& \varphi_{i}\left[ -\left(\frac{\alpha}{\beta}-1\right)(\sinh(-\varphi))^{-\frac{\alpha}{\beta}}\cosh(-\varphi)\varphi_{i}\sqrt{1+|\bar{\nabla} \varphi|^2}\sigma_{k}^{\frac{1}{\beta}} \right]\nonumber\\
			&-\varphi_{i} (\sinh(-\varphi))^{-\frac{\alpha}{\beta}+1}\sqrt{1+|\bar{\nabla} \varphi|^2}(\sigma_{k}^{\frac{1}{\beta}})_{i}
			\nonumber \\
			=& -\left(\frac{\alpha}{\beta}-1\right)
			(\sinh(-\varphi))^{-\frac{\alpha}{\beta}}\cosh(-\varphi)
			\sigma_{k}^{\frac{1}{\beta}}
			\sqrt{1+|\bar{\nabla} \varphi|^2}|\bar{\nabla}\varphi|^2 \nonumber\\
			& - (\sinh(-\varphi))^{-\frac{\alpha}{\beta}+1}\sqrt{1+|\bar{\nabla} \varphi|^2}\varphi_{i}(\sigma_{k}^{\frac{1}{\beta}})_{i},
		\end{align}
	where we used $\left(\sqrt{1+|\bar{\nabla} \varphi|^{2}}\right)_{i} = 0$ at $\left(\theta_{t}, t\right)$. For the last term, we know
		\begin{align*}
			\left(\sigma_{k}^{\frac{1}{\beta}}\right)_{i}
			=
			&\frac{1}{\beta} \sigma_{k}^{\frac{1}{\beta}-1} \frac{\partial \sigma_{k}}{\partial h_{p}^q} (h_{p}^{q})_i
			= \frac{1}{\beta} \sigma_{k}^{\frac{1}{\beta}-1} \dot{\sigma_{k}}^{p q} g_{q s} (h_{p}^s)_{i}\\
			=&\frac{1}{\beta} \sigma_{k}^{\frac{1}{\beta}-1} \dot{\sigma_{k}}^{p q} \left( (g_{q s} h_{p}^s)_{i} - (g_{q s})_{i} h_{p}^s \right) \\
			=&\frac{1}{\beta} \sigma_{k}^{\frac{1}{\beta}-1}\dot{\sigma_{k}}^{p q} \left( (h_{pq})_{i} - h_{p}^s(g_{q s})_{i} \right),
		\end{align*}
	where we regarded $h_{p q}, g_{q s}$ as tensors on $\mathbb{S}^{n}$.

	By \eqref{radial.quantities}, we know that
	$$
		h_{i j} =\left(\sqrt{\phi^{2}+|\bar{\nabla} r|^{2}}\right)^{-1}\left(-\phi \bar{\nabla}_{i} \bar{\nabla}_{j} r+2 \phi^{\prime} r_{i} r_{j}+\phi^{2} \phi^{\prime}\delta_{i j}\right).
	$$
	Since $ \bar{\nabla}_{i}\varphi = \frac{1}{\phi}\bar{\nabla}_{i}r $, we have
	$$ \bar{\nabla}_{q}\bar{\nabla}_{p}\varphi = \frac{1}{\phi}\bar{\nabla}_{q}\bar{\nabla}_{p}r - \frac{\phi'}{\phi^2}\bar{\nabla}_{q}r\bar{\nabla}_{p}r. $$
	Thus
		\begin{align*}
			h_{pq}
			&=\frac{\phi}{\sqrt{1+| \bar{\nabla}\varphi |^2}}\left(-\varphi_{pq}+\phi^{\prime} \varphi_{p} \varphi_{q}+\phi^{\prime} \delta_{pq}\right)\\
			&=\frac{1}{\sinh(-\varphi)\sqrt{1+| \bar{\nabla}\varphi |^2}}\left(-\varphi_{pq}+\frac{\cosh(-\varphi)}{\sinh(-\varphi)} \varphi_{p} \varphi_{q}+\frac{\cosh(-\varphi)}{\sinh(-\varphi)} \delta_{pq}\right).
		\end{align*}
	Hence at $\left(\theta_{t}, t\right)$ we have
		\begin{align}\label{h.pqi}
			h_{pq,i}
			\nonumber
			=&\frac{\cosh(-\varphi)}{(\sinh(-\varphi))^2}
			\frac{1}{\sqrt{1+| \bar{\nabla}\varphi|^2}}\varphi_{i}
			\left(-\varphi_{pq}+\frac{\cosh(-\varphi)}{\sinh(-\varphi)} \varphi_{p} \varphi_{q}+\frac{\cosh(-\varphi)}{\sinh(-\varphi)} \delta_{pq}\right) \\
			\nonumber
			&+\frac{1}{\sinh(-\varphi)}
			\frac{1}{\sqrt{1+| \bar{\nabla}\varphi |^2}}
			\left(-\varphi_{pqi}+ \frac{1}{(\sinh(-\varphi))^2}\varphi_{i}\varphi_{p}\varphi_{q} +\frac{1}{(\sinh(-\varphi))^2}\varphi_{i}\delta_{pq} \right.\\
			\nonumber
			&+\left.\frac{\cosh(-\varphi)}{\sinh(-\varphi)}\varphi_{pi}\varphi_{q}+\frac{\cosh(-\varphi)}{\sinh(-\varphi)}\varphi_{p}\varphi_
			{qi} \right)\\
			\nonumber
			=& \frac{\cosh(-\varphi)}{\sinh(-\varphi)}
			\varphi_{i} h_{pq}+\frac{1}{\sinh(-\varphi)}
			\frac{1}{\sqrt{1+| \bar{\nabla}\varphi |^2}}
			\left(-\varphi_{pqi}+ \frac{1}{(\sinh(-\varphi))^2}\varphi_{i}(\varphi_{p}\varphi_{q}+\delta_{pq})  \right.\\
			&\left.+ \frac{\cosh(-\varphi)}{\sinh(-\varphi)}\varphi_{pi}\varphi_{q}+\frac{\cosh(-\varphi)}{\sinh(-\varphi)}\varphi_{p}\varphi_
			{qi} \right).
		\end{align}
	Also we know that
		\begin{align*}
			g_{ij}
			=&(\phi(r))^2\delta_{i j} + r_i r_j
			=(\phi(r))^2\left( \delta_{i j} + \frac{r_i r_j}{(\phi(r))^2} \right)\\
			=&(\phi(r))^2\left( \delta_{i j} + \varphi_{i}\varphi_{j} \right)
			=\frac{1}{(\sinh(-\varphi))^2}\left( \delta_{i j} + \varphi_{i}\varphi_{j} \right).
		\end{align*}
	Hence
		\begin{align}\label{g.pqi}
			g_{pq,i}
			\nonumber
			=&\frac{2\cosh(-\varphi)}{(\sinh(-\varphi))^3}\varphi_{i}\left( \delta_{pq} + \varphi_{p}\varphi_{q} \right) +
			\frac{1}{(\sinh(-\varphi))^2}\left( \varphi_{pi}\varphi_{q} + \varphi_{p}\varphi_{qi} \right)\\
			=&\frac{2\cosh(-\varphi)}{\sinh(-\varphi)}\varphi_{i}g_{pq}+
			\frac{1}{(\sinh(-\varphi))^2}\left( \varphi_{pi}\varphi_{q} + \varphi_{p}\varphi_{qi} \right).
		\end{align}
	By using \eqref{h.pqi}, \eqref{g.pqi}  and
	$ \bar{\nabla}_p\left(\frac{1}{2}|\bar{\nabla}\varphi|^2\right) = \varphi_{i}\varphi_{pi} = 0 $ at $ (\theta_{t},t) $, we get
		\begin{align}
			\nonumber
			&\varphi_{i}\left(\sigma_{k}^{\frac{1}{\beta}}\right)_{i} =
			\frac{1}{\beta} \sigma_{k}^{\frac{1}{\beta}-1} \dot{\sigma_{k}}^{p q} \varphi_{i} \left( (h_{pq})_{i} - h_{p}^s(g_{q s})_{i} \right)\\
			=
			&\frac{1}{\beta} \sigma_{k}^{\frac{1}{\beta}-1} \dot{\sigma_{k}}^{p q} \frac{\cosh(-\varphi)}{\sinh(-\varphi)}
			h_{pq}|\bar{\nabla}\varphi|^2 - 2\frac{1}{\beta} \sigma_{k}^{\frac{1}{\beta}-1} \dot{\sigma_{k}}^{pq} h_{p}^{s}
			\frac{\cosh(-\varphi)}{\sinh(-\varphi)}g_{qs}|\bar{\nabla}\varphi|^2 \nonumber \\
			\nonumber
			&+\frac{1}{\beta} \sigma_{k}^{\frac{1}{\beta}-1} \dot{\sigma_{k}}^{p q}
			\frac{1}{\sinh(-\varphi)}
			\frac{1}{\sqrt{1+| \bar{\nabla}\varphi |^2}}
			\left(-\varphi_{pqi}\varphi_{i}+ \frac{1}{(\sinh(-\varphi))^2}(\varphi_{p}\varphi_{q}+\delta_{pq})|\bar{\nabla}\varphi|^2 \right) \\=
			\nonumber
			&-\frac{k}{\beta} \sigma_{k}^{\frac{1}{\beta}} \frac{\cosh(-\varphi)}{\sinh(-\varphi)}
			|\bar{\nabla}\varphi|^2 +
			\frac{1}{\beta} \sigma_{k}^{\frac{1}{\beta}-1}
			\frac{1}{(\sinh(-\varphi))^3} \frac{1}{\sqrt{1+| \bar{\nabla}\varphi |^2}} \dot{\sigma_{k}}^{pq}\varphi_{p}\varphi_{q}|\bar{\nabla}\varphi|^2 \\
			\label{temp1}
			&-\frac{1}{\beta} \sigma_{k}^{\frac{1}{\beta}-1}
			\frac{1}{\sinh(-\varphi)}
			\frac{1}{\sqrt{1+| \bar{\nabla}\varphi |^2}}
			\dot{\sigma_{k}}^{pq}\varphi_{i}\varphi_{pqi}.
		\end{align}
	By the Ricci identity
	$$
			\varphi_{p q i}
			=\varphi_{p i q}+\varphi_{l} R_{l p q i}^{\mathbb{S}^{n}}
			=\varphi_{i p q}+\varphi_{l}\left(\delta_{l q} \delta_{p i}-\delta_{l i} \delta_{p q}\right)
			=\varphi_{i p q}+\varphi_{q} \delta_{p i}-\varphi_{i} \delta_{p q}.
	$$
	Thus
		\begin{equation}\label{varphi.i.varphi.pqi}
			\varphi_{i} \varphi_{p q i}=\varphi_{i} \varphi_{i p q}+\varphi_{p} \varphi_{q}-\delta_{p q}|\bar{\nabla} \varphi|^{2}=\left(\frac{1}{2}|\bar{\nabla} \varphi|^{2}\right)_{p q}-\varphi_{i p} \varphi_{i q}+\varphi_{p} \varphi_{q}-\delta_{p q}|\bar{\nabla} \varphi|^{2}.
		\end{equation}
	Plugging \eqref{varphi.i.varphi.pqi} into \eqref{temp1}, we get
		\begin{align}\label{varphi.i.sigma.k.i}
			\nonumber
			&\varphi_{i}(\sigma_{k}^{\frac{1}{\beta}})_{i}
			\\=
			\nonumber
			&-\frac{k}{\beta} \sigma_{k}^{\frac{1}{\beta}} \frac{\cosh(-\varphi)}{\sinh(-\varphi)}
			|\bar{\nabla}\varphi|^2 +
			\frac{1}{\beta} \sigma_{k}^{\frac{1}{\beta}-1}
			\frac{1}{(\sinh(-\varphi))^3} \frac{1}{\sqrt{1+| \bar{\nabla}\varphi |^2}} \dot{\sigma_{k}}^{pq}(\varphi_{p}\varphi_{q}+\delta_{pq})|\bar{\nabla}\varphi|^2 \\
			&-\frac{1}{\beta} \sigma_{k}^{\frac{1}{\beta}-1}
			\frac{1}{\sinh(-\varphi)}
			\frac{1}{\sqrt{1+| \bar{\nabla}\varphi |^2}}
			\dot{\sigma_{k}}^{pq}\left[
			\left(\frac{1}{2}|\bar{\nabla}\varphi|^{2}\right)_{p q}-\varphi_{i p} \varphi_{i q}+\varphi_{p} \varphi_{q}-\delta_{p q}|\bar{\nabla} \varphi|^{2}
			\right].
		\end{align}
	Substituting \eqref{varphi.i.sigma.k.i} into \eqref{s3.C1-1} gives that
		\begin{align}
			\nonumber
			&\p_{t}\left(\frac{1}{2} |\bar{\nabla} \varphi|^{2} \right)
			\\
			=
			\nonumber
			&
			-(\frac{\alpha}{\beta}-1)(\sinh(-\varphi))^{-\frac{\alpha}{\beta}}\cosh(-\varphi)\sigma_{k}^{\frac{1}{\beta}}\sqrt{1+|\bar{\nabla} \varphi|^2}|\bar{\nabla}\varphi|^2 \\
			\nonumber
			&+(\sinh(-\varphi))^{-\frac{\alpha}{\beta}+1}\sqrt{1+|\bar{\nabla} \varphi|^2}\cdot
			\frac{k}{\beta} \sigma_{k}^{\frac{1}{\beta}-1} \frac{\cosh(-\varphi)}{\sinh(-\varphi)}
			\sigma_{k}
			|\bar{\nabla}\varphi|^2 \\
			\nonumber
			 &-(\sinh(-\varphi))^{-\frac{\alpha}{\beta}+1}\sqrt{1+|\bar{\nabla} \varphi|^2}
			 \cdot
			 \frac{1}{\beta} \sigma_{k}^{\frac{1}{\beta}-1}
			\frac{1}{(\sinh(-\varphi))^3} \frac{1}{\sqrt{1+| \bar{\nabla}\varphi |^2}} \dot{\sigma_{k}}^{pq}(\varphi_{p}\varphi_{q}+\delta_{pq})|\bar{\nabla}\varphi|^2 \\
			\nonumber
			&+(\sinh(-\varphi))^{-\frac{\alpha}{\beta}+1}\sqrt{1+|\bar{\nabla} \varphi|^2}
			\cdot
			\frac{1}{\beta} \sigma_{k}^{\frac{1}{\beta}-1}
			\frac{1}{\sinh(-\varphi)}
			\frac{1}{\sqrt{1+| \bar{\nabla}\varphi |^2}}
			\dot{\sigma_{k}}^{pq}\\
			\nonumber
			&\times \left[
			\left(\frac{1}{2}|\bar{\nabla} \varphi|^{2}\right)_{p q}-\varphi_{i p} \varphi_{i q}+\varphi_{p} \varphi_{q}-\delta_{p q}|\bar{\nabla} \varphi|^{2}
			\right]\\
			=
			\nonumber
			&\frac{1}{\beta}(\sinh(-\varphi))^{-\frac{\alpha}{\beta}} \sigma_{k}^{\frac{1}{\beta}-1}\dot{\sigma_{k}}^{pq}\left(\frac{1}{2}|
			\bar{\nabla} \varphi|^{2}\right)_{p q} -
			\frac{1}{\beta}(\sinh(-\varphi))^{-\frac{\alpha}{\beta}-2}
			\sigma_{k}^{\frac{1}{\beta}-1} \dot{\sigma_{k}}^{pq}\varphi_{p}\varphi_{q}|\bar{\nabla}\varphi|^2\\
			\nonumber
			&-
			\frac{1}{\beta}(\sinh(-\varphi))^{-\frac{\alpha}{\beta}-2}
			\sigma_{k}^{\frac{1}{\beta}-1} \dot{\sigma_{k}}^{pq}\delta_{pq}|\bar{\nabla}\varphi|^2\\
			\nonumber
			&-\frac{\alpha-k-\beta}{\beta}(\sinh(-\varphi))^{-\frac{\alpha}{\beta}}\cosh(-\varphi)\sigma_{k}^{\frac{1}{\beta}}\sqrt{1+|\bar{\nabla} \varphi|^2}|\bar{\nabla}\varphi|^2 \\
			\nonumber
			&-\frac{1}{\beta}(\sinh(-\varphi))^{-\frac{\alpha}{\beta}}
			\sigma_{k}^{\frac{1}{\beta}-1}
			\dot{\sigma_{k}}^{pq}\varphi_{i p}\varphi_{i q}
			- \frac{1}{\beta}(\sinh(-\varphi))^{-\frac{\alpha}{\beta}}
			\sigma_{k}^{\frac{1}{\beta}-1}
			\dot{\sigma_{k}}^{pq}
			\left( \delta_{p q}|\bar{\nabla} \varphi|^{2}-\varphi_{p} \varphi_{q} \right) \\
			\leq
			\label{p.nabla.varphi}
			&0,
		\end{align}
	where the last inequality is from $\left(|\bar{\nabla} \varphi|^{2}\right)_{p q} \leq 0$ at $\left(\theta_{t}, t\right)$, $(\dot{\sigma_{k}}^{p q} )\geq 0$ and $(\delta_{p q}|\bar{\nabla}  \varphi|^{2}-\varphi_{p} \varphi_{q} )\geq 0$.
	
	When $\alpha > k+\beta$, we get $\frac{d}{dt}\left(\frac{1}{2}|\bar{\nabla} \varphi|^{2}\right)_{\max}(t) \leq 0$, so $|\bar{\nabla} \varphi|$ is bounded from above by a constant. Since $ \bar{\nabla}_{i}r = \sinh(r)\cdot\bar{\nabla}_{i}\varphi $, we have $|\bar{\nabla} r| \leq C$ for some constant $C$.
\end{proof}
\begin{cor}\label{C1estimate2}
	Under the same assumptions in \lemref{C1estimate}, along the flow \eqref{flow}, the hypersurface $\mathcal{M}_{t}$ preserves star-shapedness and the support function u satisfies
	$$
	\frac{1}{C} \leq u \leq C
	$$
	for some constant $C>0$.
\end{cor}
\begin{proof}
	We know $ u=\frac{\phi}{\omega} = \frac{\phi(r)}{\sqrt{1+|\bar{\nabla} \varphi|^{2}}} $.
	Since $ r $ is bounded from both above and below, we only need to prove $ \omega = \sqrt{1+|\bar{\nabla} \varphi|^{2}} $ is bounded from both above and below. This follows from  $ \omega = \sqrt{1+|\bar{\nabla} \varphi|^{2}} \geq 1 $ and \lemref{C1estimate} immediately.
\end{proof}

\section{Estimates for $ k $-curvature}\label{sec:4}
In the following sections, we will derive the $C^{2}$ estimates of flow \eqref{flow}. To simplify the statements of following lemmas, we always assume $\alpha > k+\beta$ and the initial hypersurface is smooth, closed, star-shaped and $k$-convex. Besides, we denote $ F = (\sigma_{k})^{\frac{1}{\beta}} $ and we define a parabolic operator
$$
\mathcal{L}=\partial_{t}-\frac{1}{\beta} \sigma_{k}^{\frac{1}{\beta}-1}(\sinh(r))^{\frac{\alpha}{\beta}} \dot{\sigma_{k}}^{i j} \nabla_{i} \nabla_{j}
=\partial_{t}-(\phi(r))^{\frac{\alpha}{\beta}} \dot{F}^{i j} \nabla_{i} \nabla_{j}.
$$

In this section, we prove that the $ k $-curvature is bounded from both above and below. We first calculate the evolution equations of $\Phi=\phi^{\alpha/\beta}F$ and the support function $u$.
\begin{lem}
Along the flow \eqref{flow}, we have
\begin{align}\label{partial.t.Phi}
\mathcal{L}\Phi=&\frac{\alpha-k}{\beta}\gamma\phi'(r)\Phi - \frac{\alpha}{\beta}\omega\frac{\phi'(r)}{\phi(r)}\Phi^2 +(\phi(r))^{\frac{\alpha}{\beta}}\frac{1}{\beta} \sigma_{k}^{\frac{1}{\beta}-1} \Phi \left(\sigma_1\sigma_{k}-(k+1)\sigma_{k+1}\right)\nonumber \\
		&+\frac{n-k+1}{\beta} \sigma_{k}^{\frac{1}{\beta}-1}(\phi(r))^{\frac{\alpha}{\beta}} (\gamma u-\Phi) \sigma_{k-1}.
\end{align}
and
\begin{align}\label{L.u}
		\mathcal{L}u
		=&\gamma\phi'(r)u - \frac{k+\beta}{\beta}\phi'(r) \Phi + \frac{\alpha}{\beta}\phi'(r)|\nabla r|^2\Phi +
		\frac{1}{\beta}\sigma_{k}^{\frac{1}{\beta}-1}
		(\phi(r))^{\frac{\alpha}{\beta}} \left(
		\sigma_1\sigma_{k}-(k+1)\sigma_{k+1}
		\right)u.
	\end{align}
\end{lem}
\proof
First, we calculate the evolution equation of $ \Phi $. From \corref{evolve.quantities1},
	\begin{align*}
		\p_{t} \Phi =
		& \p_{t}\left( (\phi(r))^{\frac{\alpha}{\beta}} \sigma_{k}^{\frac{1}{\beta}} \right)  \\=
		& \frac{\alpha}{\beta} (\phi(r))^{\frac{\alpha}{\beta}-1} \phi'(r) (\p_{t}r) \sigma_{k}^{\frac{1}{\beta}} +
		(\phi(r))^{\frac{\alpha}{\beta}} \frac{\p F}{\p h_{i}^{j}}
		\p_{t} h_{i}^{j} \\=
		& \frac{\alpha}{\beta} (\phi(r))^{\frac{\alpha}{\beta}-1} \phi'(r) (-\Phi\omega + \gamma \phi(r))\sigma_{k}^{\frac{1}{\beta}}  \\
		&+(\phi(r))^{\frac{\alpha}{\beta}}
		\frac{\p F}{\p h_{i}^{j}}
		\left( \nabla_{i}\nabla^{j}\Phi+\Phi h_{i}^{k} h_{k}^{j} -
		\gamma\phi'h_{i}^{j} + (\gamma u -\Phi)\delta_{i}^{j} \right)\\=
		&\frac{\alpha}{\beta} \Phi \frac{\phi'(r)}{\phi(r)}(-\Phi \omega + \gamma \phi(r)) \\
		&+(\phi(r))^{\frac{\alpha}{\beta}}
		\dot{F}^{ij} \left( \nabla_{i}\nabla_{j}\Phi+\Phi (h^2)_{ij} - \gamma\phi'h_{ij} + (\gamma u -\Phi)g_{ij} \right).
	\end{align*}
	By using \corref{F.trace}, we have
	\begin{align*}
		\nonumber
		\p_{t} \Phi
		=
		\nonumber
		&\frac{\alpha}{\beta}\gamma\phi'(r)\Phi - \frac{\alpha}{\beta}\omega\frac{\phi'(r)}{\phi(r)}\Phi^2 +
		(\phi(r))^{\frac{\alpha}{\beta}}\dot{F}^{ij}\nabla_{i}\nabla_{j}\Phi\\
		\nonumber
		&+\frac{1}{\beta} \sigma_{k}^{\frac{1}{\beta}-1}
		(\phi(r))^{\frac{\alpha}{\beta}} \Phi \left(\sigma_1\sigma_{k}-(k+1)\sigma_{k+1}\right)
		-
		\frac{k\gamma}{\beta} \sigma_{k}^{\frac{1}{\beta}} (\phi(r))^{\frac{\alpha}{\beta}}\phi'(r)\\
		\nonumber
		&+
		\frac{n-k+1}{\beta} \sigma_{k}^{\frac{1}{\beta}-1}(\phi(r))^{\frac{\alpha}{\beta}} (\gamma u-\Phi) \sigma_{k-1}\\
		=
		\nonumber
		&(\phi(r))^{\frac{\alpha}{\beta}}\dot{F}^{ij}
		\nabla_{i}\nabla_{j}\Phi +
		\frac{\alpha-k}{\beta}\gamma\phi'(r)\Phi - \frac{\alpha}{\beta}\omega\frac{\phi'(r)}{\phi(r)}\Phi^2 \\
		\nonumber
		&+(\phi(r))^{\frac{\alpha}{\beta}}\frac{1}{\beta} \sigma_{k}^{\frac{1}{\beta}-1} \Phi \left(\sigma_1\sigma_{k}-(k+1)\sigma_{k+1}\right) \\
		&+\frac{n-k+1}{\beta} \sigma_{k}^{\frac{1}{\beta}-1}(\phi(r))^{\frac{\alpha}{\beta}} (\gamma u-\Phi) \sigma_{k-1},
	\end{align*}
	which implies \eqref{partial.t.Phi}.

	Next we calculate the evolution of $u$. Note that the covariant derivatives of $u$ satisfy (see Lemma 2.6 of \cite{Guan-2014})
	$$
	\nabla_{i} u
	=h_{i}^{k}\left\langle V, X_{k}\right\rangle
	$$
and
	\begin{equation}\label{temp11}
		\begin{aligned}
			\nabla_{i} \nabla_{j} u &=\nabla^{k} h_{i j}\left\langle V, X_{k}\right\rangle+\phi^{\prime} h_{i j}-h_{i k} h_{j}^{k} u .
		\end{aligned}
	\end{equation}
	By \eqref{evolve.quantities}, we have
	\begin{equation}\label{temp12}
		\frac{\partial}{\partial t} u
				= \phi'(r)(\gamma u-\Phi) + \langle V,\nabla\Phi\rangle .
	\end{equation}
	Also we have
	\begin{align*}
		\nabla_{i}\Phi
		&= \nabla_{i}\left( (\phi(r))^{\frac{\alpha}{\beta}} \sigma_{k}^{\frac{1}{\beta}} \right)\\
		&= \frac{\alpha}{\beta}(\phi(r))^{\frac{\alpha}{\beta}-1}\phi'(r)(\nabla_{i}r)
		\sigma_{k}^{\frac{1}{\beta}} +
		(\phi(r))^{\frac{\alpha}{\beta}} \nabla_{i}\left(\sigma_{k}^{\frac{1}{\beta}}\right)\\
		&=\frac{\alpha}{\beta}\frac{\phi'(r)}{\phi(r)}\Phi\nabla_{i}r +
		(\phi(r))^{\frac{\alpha}{\beta}} F^{pq}\nabla_{i}h_{pq}.
	\end{align*}
	By using $ \nabla_{j}r = \frac{\left\langle X_j,V \right\rangle}{\phi(r)} $, we have
	\begin{align}
		\langle V,\nabla\Phi\rangle
		\nonumber
		&=g^{ij}(\nabla_{i}\Phi)\langle V,X_j\rangle\\
		\nonumber
		&=\frac{\alpha}{\beta}\phi'(r)|\nabla r|^2\Phi +
		(\phi(r))^{\frac{\alpha}{\beta}} g^{ij} F^{pq}\nabla_{i}h_{pq}\langle V,X_j\rangle \\
		\label{temp13}
		&=\frac{\alpha}{\beta}\phi'(r)|\nabla r|^2\Phi +
		(\phi(r))^{\frac{\alpha}{\beta}} F^{pq}\nabla^{k}h_{pq}\langle V,X_k\rangle.
	\end{align}
	Combining \eqref{temp11}, \eqref{temp12} and \eqref{temp13}, and by using \corref{F.trace}, we get
	\begin{align*}
		\mathcal{L}u
		\nonumber
		=&\p_{t} u - (\phi(r))^{\frac{\alpha}{\beta}} \dot{F}^{pq}\nabla_p\nabla_q u \\
		\nonumber
		=&\phi'(r)(\gamma u-\Phi) + \frac{\alpha}{\beta}\phi'(r)|\nabla r|^2\Phi -
		\phi'(r)(\phi(r))^{\frac{\alpha}{\beta}} \dot{F}^{pq}h_{pq} +
		(\phi(r))^{\frac{\alpha}{\beta}} \dot{F}^{pq}(h^2)_{pq}u \\
		\nonumber
		=&\phi'(r)(\gamma u-\Phi) - \frac k{\beta}\phi'(r) \Phi + \frac{\alpha}{\beta}\phi'(r)|\nabla r|^2\Phi +
		\frac{1}{\beta}\sigma_{k}^{\frac{1}{\beta}-1}
		(\phi(r))^{\frac{\alpha}{\beta}} \left(
		\sigma_1\sigma_{k}-(k+1)\sigma_{k+1}
		\right)u \\
		=&\gamma\phi'(r)u - \frac{k+\beta}{\beta}\phi'(r) \Phi + \frac{\alpha}{\beta}\phi'(r)|\nabla r|^2\Phi +
		\frac{1}{\beta}\sigma_{k}^{\frac{1}{\beta}-1}
		(\phi(r))^{\frac{\alpha}{\beta}} \left(
		\sigma_1\sigma_{k}-(k+1)\sigma_{k+1}
		\right)u.
	\end{align*}
\endproof

\begin{prop}\label{C2estimate1}
	Along the flow \eqref{flow}, if the initial hypersurface $ \MM_{0} $ is smooth, closed and $ k $-convex, then there exists a constant $ C > 0 $ such that
	$$ \sigma_{k} \geq C $$
on $\MM_t$ for $t\in [0,T)$.
\end{prop}
\begin{proof}
	By \lemref{Maclaurin}, we know that
	$ \sigma_1\sigma_{k}-(k+1)\sigma_{k+1} \geq C_1(\sigma_{k})^{1+\frac{1}{k}} > 0 $ for some constant $ C_1=C_1(n,k) >0$, and $ \sigma_{k-1} \geq C_2(\sigma_{k})^{\frac{k-1}{k}} >0 $ for some constant $ C_2=C_2(n,k) >0 $.
	
	We apply maximum principle to the evolution equation \eqref{partial.t.Phi} of $\Phi$. At the spatial minimum point of $ \Phi $ on $ \mathcal{M}_t $, we have $ (\nabla_{i}\nabla_{j}\Phi )\geq 0 $,
	hence $ \dot{\sigma_{k}}^{ij}
	\nabla_{i}\nabla_{j}\Phi \geq 0 $,
	then $ \dot{F}^{ij}\nabla_{i}\nabla_{j}\Phi \geq 0 $.
	If $ \Phi $ is small enough such that $ \Phi \leq \gamma \min_{\mathbb{S}^{n}\times[0,T)}u $, we have
	$$
	\p_{t}\Phi
	\geq \frac{\alpha-k}{\beta}\gamma\phi'(r)\Phi - \frac{\alpha}{\beta}\omega\frac{\phi'(r)}{\phi(r)}\Phi^2
	=\frac{\alpha-k}{\beta}\gamma\phi'(r)\Phi -
	\frac{\alpha}{\beta}\frac{\phi'(r)}{u}\Phi^2
	$$
at the spatial minimum point of $ \Phi $.	By \lemref{C0estimate} and \lemref{C1estimate} ,let
	$$ C_1 = \frac{\alpha}{(\alpha-k)\gamma u_{\min}}. $$
	we have
	$$
	\p_{t}\Phi \geq \frac{\alpha-k}{\beta}\gamma\phi'(r) \left(\Phi - C_1\Phi^2\right).
	$$
	
	Therefore, if $ 0\leq\Phi_{\min}(t) \leq \min\left\{ \frac{1}{C_1}, \gamma \min_{\mathbb{S}^{n}\times[0,T)}u \right\} $, then
	$ \frac{d}{dt} \Phi_{\min}(t) \geq 0 $.
	In all, we conclude that
	$ \Phi $ is bounded from below by a positive constant.
	Since $ \sigma_{k}^{\frac{1}{\beta}} = \frac{1}{(\phi(r))^{\frac{\alpha}{\beta}}}\Phi $, we know that $ \sigma_{k} $ is bounded from below by a positive constant $ C $.
\end{proof}
\begin{prop}\label{C2estimate2}
	Along the flow \eqref{flow}, if the initial hypersurface $ \MM_{0} $ is smooth, closed and $ k $-convex,  then there exists a constant $C>0$, such that
	$$
	\sigma_{k} \leq C
	$$
	on $\MM_t$ for $t\in [0,T)$.
\end{prop}
\begin{proof}
	To prove the upper bound of $\sigma_k$, we apply the technique of Tso \cite{Chou85}. We consider the auxiliary function $Q=\log \Phi-\log (u-a)$, where $a=\frac{1}{2} \inf _{\SSS^n \times[0, T)} u$. Using the equations \eqref{partial.t.Phi} and \eqref{L.u}, at the maximum point of $ Q $ on $ \mathcal{M}_{t} $, we have
	\begin{align*}
		&\mathcal{L}Q
		=\frac{\mathcal{L}\Phi}{\Phi} - \frac{\mathcal{L}u}{u-a}\\
		=&\frac{\alpha-k}{\beta}\gamma\phi'(r) - \frac{\alpha}{\beta}\omega\frac{\phi'(r)}{\phi(r)}\Phi +
		(\phi(r))^{\frac{\alpha}{\beta}}
		\frac{1}{\beta} \sigma_{k}^{\frac{1}{\beta}-1} \left(\sigma_1\sigma_{k}-(k+1)\sigma_{k+1}\right)\\
		&+
		\frac{n-k+1}{\beta} \sigma_{k}^{\frac{1}{\beta}-1}(\phi(r))^{\frac{\alpha}{\beta}} \frac{(\gamma u-\Phi)}{\Phi} \sigma_{k-1}
		-\gamma\phi'(r)\frac{u}{u-a} +
		\frac{(k+\beta)\phi'(r)}{\beta(u-a)} \Phi \\
		&-
		\frac{\alpha}{\beta}\frac{\phi'(r)|\nabla r|^2}{u-a}\Phi -
		(\phi(r))^{\frac{\alpha}{\beta}}\frac{1}{\beta} \sigma_{k}^{\frac{1}{\beta}-1} \left(
		\sigma_1\sigma_{k}-(k+1)\sigma_{k+1}
		\right)\frac{u}{u-a} \\
		\leq & -\frac{a}{u-a}
		\frac{1}{\beta} \sigma_{k}^{\frac{1}{\beta}-1}
		(\phi(r))^{\frac{\alpha}{\beta}}  \left(\sigma_1\sigma_{k}-(k+1)\sigma_{k+1}\right) + C\Phi +C \\
		=&-\frac{a}{\beta(u-a)}
		\frac{\sigma_1\sigma_{k}-(k+1)\sigma_{k+1}}{\sigma_k}\Phi + C\Phi +C
	\end{align*}
	for some $ C > 0 $ by \lemref{C0estimate} and \lemref{C1estimate} and \corref{C1estimate2}, upon assuming that
	$ \Phi $ is large enough such that $ \Phi\geq \gamma \max_{\mathbb{S}^{n}\times[0,T)}u $.
	%
	By \lemref{Maclaurin},
	we can estimate
	$$ \frac{\sigma_1\sigma_{k}-(k+1)\sigma_{k+1}}{\sigma_k} \geq \frac{k}{ \tbinom{n}{k}^{\frac{1}{k}}}(\sigma_{k})^{\frac{1}{k}}.  $$
	
	Thus
	\begin{equation}
		\mathcal{L}Q \leq -C_1\Phi^{1+\frac{\beta}{k}} + C\Phi+C.
	\end{equation}
	for some constant $ C,C_1>0 $. Thus, there exists a constant $C>0$ that only depend on $\mathcal{M}_{0}$, so that whenever $\Phi>C$, we have $\frac{d}{d t} Q_{\max }(t)<0 .$ From \lemref{C0estimate} and \lemref{C1estimate}, $r$ and $ u $ is bounded. Hence $\Phi$ goes to infinity when $Q$ goes to infinity. Therefore $Q$ is bounded from above by a constant, which gives an upper bound of $\Phi$.
\end{proof}

\section{$C^2$ estimate for $k=1$ and mean convex solutions}\label{sec:5}
In this section, we consider the case $k=1$ of the flow \eqref{flow}. Since $0<\beta\leq 1$, the flow is fully nonlinear except $\beta=1$, we still need to get $C^2$ estimate. We will prove the boundness of principal curvatures when the initial hypersurface is mean convex. We first calculate the evolution equation of $|A|^2$, where $ |A| $ denotes the tensor $ \left(h_{i}^{j}\right) $ and $ |A|^2 = h_{i}^{j} h_{j}^{i} .$
\begin{lem}
Along the flow \eqref{flow},
\begin{align}
		\nonumber
		\mathcal{L}\left( |A|^2  \right)
		=
		\nonumber
		&- 2\frac{1}{\beta} \phi^{\frac{\alpha}{\beta}} H^{\frac{1}{\beta}-1}|\nabla A|^2 +
		2(\phi(r))^{\frac{\alpha}{\beta}} (1-\beta) H^{-\frac{1}{\beta}} h^{ij}\nabla_{i} F \nabla_{j} F\\
		\nonumber
		& +
		2 \frac{\alpha}{\beta}\left(\frac{\alpha}{\beta}-1\right)\frac{(\phi'(r))^2}{(\phi(r))^2}\Phi
		h^{ij}(\nabla_{i}r)(\nabla_{j}r)
		-
		2 \frac{\alpha}{\beta}\frac{1}{(\phi(r))^2}\Phi h^{ij}(\nabla_{i}r)(\nabla_{j}r)
		\\
		\nonumber
		&-
		2 \frac{\alpha}{\beta}\frac{\phi'(r)}{(\phi(r))^2}u\Phi |A|^2
		+ 2 \frac{\alpha}{\beta}\frac{(\phi'(r))^2}{(\phi(r))^2}\Phi H
		+
		4\frac{\alpha}{\beta}\phi'(r) (\phi(r))^{\frac{\alpha}{\beta}-1}h^{ij}
		(\nabla_{i}r)(\nabla_{j}F)\\
		&+
		\frac{2}{\beta}H^{\frac{1}{\beta}-1}(\phi(r))^{\frac{\alpha}{\beta}}|A|^4
		-\frac{2}{\beta}\Phi H + \frac{2n}{\beta}H^{\frac{1}{\beta}-1}(\phi(r))^{\frac{\alpha}{\beta}}|A|^2
		\label{L.A2}\\
&		-
		4 \gamma \phi'(r) |A|^2 +
		2 (\gamma u - \Phi) H+
		2 (\gamma u - \Phi) H+2(1-\frac 1{\beta})\Phi h_i^lh_l^jh_j^i.\nonumber
	\end{align}
\end{lem}
\proof
By Corollary \ref{evolve.quantities1}, we have
		\begin{align}
			\nonumber
			\frac{\p}{\p t}|A|^2
			=& 2 h_{j}^{i}\frac{\p}{\p t}h_{i}^{j}
			=2 h_{j}^{i}\left(
			\nabla_{i}\nabla^{j}\Phi+\Phi h_{i}^{k} h_{k}^{j} - \gamma\phi'h_{i}^{j}+(\gamma u -\Phi)\delta_{i}^{j}
			\right)\\
			\label{tmp1}
			=& 2 h^{ij}\nabla_{i}\nabla_{j}\Phi +
			2 \Phi h_{i}^{l} h_{l}^{j} h_{j}^{i} -
			4 \gamma \phi' |A|^2 +
			2  (\gamma u - \Phi) H.
		\end{align}
	Since $ k=1 $, $\sigma_{k} = H$, $F=H^{\frac{1}{\beta}}, \mathcal{L}=\partial_{t}-\frac{1}{\beta} \phi^{\frac{\alpha}{\beta}} H^{\frac{1}{\beta}-1} g^{ij}\nabla_{i}\nabla_{j} $. We observe that
	$$
	\dot{F}^{p q}=\frac{1}{\beta} H^{\frac{1}{\beta}-1} g^{p q}, \quad \ddot{F}^{p q, r s}=\frac{1-\beta}{\beta^{2}} H^{\frac{1}{\beta}-2} g^{r s} g^{p q} .
	$$
	Thus
	$$
	\ddot{F}^{p q, r s} \nabla_{i} h_{p q} \nabla_{j} h_{r s}=(1-\beta) H^{-\frac{1}{\beta}} \nabla_{i} F \nabla_{j} F.
	$$
	By \eqref{bigresult}, we have
		\begin{align}
			\nabla_{(i}\nabla_{j)}\Phi
			\nonumber
			=&\frac{1}{\beta} (\phi(r))^{\frac{\alpha}{\beta}} H^{\frac{1}{\beta}-1} g^{p q}\nabla_{p}\nabla_{q}h_{ij}
			+  (\phi(r))^{\frac{\alpha}{\beta}} (1-\beta) H^{-\frac{1}{\beta}} \nabla_{i} F \nabla_{j} F \\
			\nonumber
			&+\left(\frac{\alpha}{\beta}\right)^2\frac{(\phi'(r))^2}{(\phi(r))^2}(\nabla_{i}r)(\nabla_{j}r)
			\Phi -
			\frac{\alpha}{\beta}\frac{1}{(\phi(r))^2}(\nabla_{i}r)(\nabla_{j}r)\Phi -
			\frac{\alpha}{\beta}\frac{\phi'(r)}{(\phi(r))^2}uh_{ij}\Phi \\
			\nonumber
			&+\frac{\alpha}{\beta}\frac{(\phi'(r))^2}{(\phi(r))^2}
			\left(
			g_{ij}-(\nabla_{i}r)(\nabla_{j}r)
			\right)\Phi
			+2\frac{\alpha}{\beta}\phi'(r) (\phi(r))^{\frac{\alpha}{\beta}-1}
			(\nabla_{(i}r)(\nabla_{j)}F)
			\\
			&+
			\frac{1}{\beta}(\phi(r))^{\frac{\alpha}{\beta}}H^{\frac{1}{\beta}-1}
			|A|^2 h_{ij}
			-\frac{1}{\beta}\Phi\left(h^2\right)_{ij}
			\label{tmp2}
			-\frac{1}{\beta}\Phi g_{ij} +
			\frac{n}{\beta}H^{\frac{1}{\beta}-1}(\phi(r))^{\frac{\alpha}{\beta}}h_{ij}.
		\end{align}
	Plugging \eqref{tmp2} into \eqref{tmp1},
	we get
	\begin{align*}
		\frac{\p}{\p t}|A|^2
		=& \frac{2}{\beta}(\phi(r))^{\frac{\alpha}{\beta}}H^{\frac{1}{\beta}-1}
		g^{pq}h^{ij}\nabla_{p}\nabla_{q}h_{ij} +
		2(\phi(r))^{\frac{\alpha}{\beta}} (1-\beta) H^{-\frac{1}{\beta}} h^{ij}\nabla_{i} F \nabla_{j} F\\
		& +
		2 \frac{\alpha}{\beta}\left(\frac{\alpha}{\beta}-1\right)\frac{(\phi'(r))^2}{(\phi(r))^2}\Phi
		h^{ij}(\nabla_{i}r)(\nabla_{j}r)
		-
		2 \frac{\alpha}{\beta}\frac{1}{(\phi(r))^2}\Phi h^{ij}(\nabla_{i}r)(\nabla_{j}r)
		\\
		&-
		2 \frac{\alpha}{\beta}\frac{\phi'(r)}{(\phi(r))^2}u\Phi |A|^2
		+ 2 \frac{\alpha}{\beta}\frac{(\phi'(r))^2}{(\phi(r))^2}\Phi H
		+
		4\frac{\alpha}{\beta}\phi'(r) (\phi(r))^{\frac{\alpha}{\beta}-1}h^{ij}
		(\nabla_{i}r)(\nabla_{j}F)\\
		&+
		\frac{2}{\beta}H^{\frac{1}{\beta}-1}(\phi(r))^{\frac{\alpha}{\beta}}|A|^4
		-\frac{2}{\beta}\Phi H + \frac{2n}{\beta}H^{\frac{1}{\beta}-1}(\phi(r))^{\frac{\alpha}{\beta}}|A|^2 \\
		&-
		4 \gamma \phi'(r) |A|^2 +
		2 (\gamma u - \Phi) H+2(1-\frac 1{\beta})\Phi h_i^lh_l^jh_j^i.
	\end{align*}
	Also,
	\begin{align*}
		g^{pq}\nabla_{p}\nabla_{q}|A|^2
		=& g^{pq}\nabla_{p}\left( 2h_{j}^{i}\nabla_{q}h_{i}^{j} \right)
		= 2g^{pq}\left(\nabla_{p}h_{j}^{i}\right)\left(\nabla_{q}h_{i}^{j}\right) +
		2g^{pq}h_{j}^{i}\nabla_{p}\nabla_{q}h_{i}^{j}\\
		=&2g^{pq}\left(\nabla_{p}h_{j}^{i}\right)\left(\nabla_{q}h_{i}^{j}\right) +
		2g^{pq}h^{ij}h_{ij,pq}
		=2|\nabla A|^2 + 2g^{pq}h^{ij}h_{ij,pq}.
	\end{align*}
	Combining the above two equations, we get \eqref{L.A2}.
\endproof

\begin{prop}\label{C2estimate3}
	Under the flow \eqref{flow} for $ k =1 $, when the initial hypersurface is mean convex, if $\alpha > 1+\beta$, the principal curvatures of the mean convex solution have a uniform bound, i.e.
	$$
	\left|\kappa_{i}\right| \leq C \quad i=1, \cdots, n.
	$$
\end{prop}
\begin{proof}
	Plugging $ k=1 $, $ \Phi = (\phi(r))^{\frac{\alpha}{\beta}}H $	into \eqref{partial.t.Phi}, we get
	\begin{align}\label{L.Phi1}
		\mathcal{L}\Phi=&
		\frac{\alpha-1}{\beta}\gamma\phi'(r)\Phi - \frac{\alpha}{\beta}\omega\frac{\phi'(r)}{\phi(r)}\Phi^2
		+\frac{1}{\beta}(\phi(r))^{\frac{\alpha}{\beta}}  H^{\frac{1}{\beta}-1}
		\Phi |A|^2\nonumber\\
&		+
		\frac{n}{\beta} H^{\frac{1}{\beta}-1}(\phi(r))^{\frac{\alpha}{\beta}} (\gamma u-\Phi).
	\end{align}
Inspired by the work \cite{L-X-Z-2021} of Li, Xu and Zhang, we define an auxiliary function $$ Q = \log|A|^2 - 2B\log (\Phi-a), $$ where
	$ B=1-\frac{a}{2c} $,
	$ a=\frac{1}{2} \inf_{\SSS^n \times[0, T)} \Phi $, and
	$ c=\sup_{\SSS^n \times[0, T)} \Phi $.
	Thus
		\begin{align}\label{L.Q.temp}
			&\mathcal{L}Q
			=\mathcal{L}\biggl( \log(|A|^2) - 2B\log (\Phi-a) \biggr)\nonumber\\
			=&\frac{\mathcal{L}\left( |A|^2 \right)}{|A|^2} +
			\frac{1}{\beta} \phi^{\frac{\alpha}{\beta}} H^{\frac{1}{\beta}-1}g^{ij}\frac{\nabla_{i}|A|^2}{|A|^2}\frac{\nabla_{j}|A|^2}{|A|^2} -
			2B\frac{\mathcal{L}\Phi}{\Phi-a} \nonumber\\
&-
			2B \frac{1}{\beta} \phi^{\frac{\alpha}{\beta}} H^{\frac{1}{\beta}-1}g^{ij} \frac{\nabla_{i}\Phi}{\Phi-a}\frac{\nabla_{j}\Phi}{\Phi-a}.
		\end{align}
	At the spatial maximum point of $ Q $, we have
	\begin{equation}\label{L.Q.condition}
		\frac{\nabla_{i}|A|^2}{|A|^2} = 2B\frac{\nabla_{i}\Phi}{\Phi-a}.
	\end{equation}
	Plugging \eqref{L.A2}, \eqref{L.Phi1} and \eqref{L.Q.condition} into \eqref{L.Q.temp}, we get
	\begin{align}\label{L.Q.temp1}
		\nonumber
		&\mathcal{L}Q
		\\
		\nonumber
		=&\frac{\mathcal{L}\left( |A|^2 \right)}{|A|^2} +
		4B^2\frac{1}{\beta} \phi^{\frac{\alpha}{\beta}} H^{\frac{1}{\beta}-1} g^{ij} \frac{\nabla_{i}\Phi}{\Phi-a}\frac{\nabla_{j}\Phi}{\Phi-a} -
		2B\frac{\mathcal{L}\Phi}{\Phi-a} -
		2B \frac{1}{\beta} \phi^{\frac{\alpha}{\beta}} H^{\frac{1}{\beta}-1}g^{ij}
		\frac{\nabla_{i}\Phi}{\Phi-a}\frac{\nabla_{j}\Phi}{\Phi-a}\\
		=\nonumber
		&- \frac{2}{\beta} \phi^{\frac{\alpha}{\beta}} H^{\frac{1}{\beta}-1}
		\frac{|\nabla A|^2}{|A|^2} +
		2(\phi(r))^{\frac{\alpha}{\beta}} (1-\beta) H^{-\frac{1}{\beta}} \frac{h^{ij}}{|A|^2}\nabla_{i} F \nabla_{j} F\\
		\nonumber
		& +
		2 \frac{\alpha}{\beta}\left(\frac{\alpha}{\beta}-1\right)\frac{(\phi'(r))^2}{(\phi(r))^2}\Phi
		\frac{h^{ij}}{|A|^2}(\nabla_{i}r)(\nabla_{j}r)
		-
		2 \frac{\alpha}{\beta}\frac{1}{(\phi(r))^2}\Phi \frac{h^{ij}}{|A|^2}(\nabla_{i}r)(\nabla_{j}r)
		\\
		\nonumber
		&-
		2 \frac{\alpha}{\beta}\frac{\phi'(r)}{(\phi(r))^2}u\Phi
		+ 2 \frac{\alpha}{\beta}\frac{(\phi'(r))^2}{(\phi(r))^2}\Phi \frac{H}{|A|^2}
		+
		4\frac{\alpha}{\beta}\phi'(r) (\phi(r))^{\frac{\alpha}{\beta}-1}
		\frac{h^{ij}}{|A|^2}
		(\nabla_{i}r)(\nabla_{j}F)\\
		\nonumber
		&+
		\frac{2}{\beta}H^{\frac{1}{\beta}-1}(\phi(r))^{\frac{\alpha}{\beta}}|A|^2
		-\frac{2}{\beta}\Phi \frac{H}{|A|^2} + \frac{2n}{\beta}H^{\frac{1}{\beta}-1}(\phi(r))^{\frac{\alpha}{\beta}}
		-
		4 \gamma \phi'(r)  +
		2 (\gamma u - \Phi) \frac{H}{|A|^2}\\
		\nonumber
		&
		+2(1-\frac 1{\beta}) \frac{h_i^lh_l^jh_j^i}{|A|^2}\Phi
		-\frac{2B(\alpha-1)}{\beta}\gamma\phi'(r)\frac{\Phi}{\Phi-a} + \frac{2B\alpha}{\beta}\omega\frac{\phi'(r)}{\phi(r)}\frac{\Phi^2}{\Phi-a}\\
		\nonumber
		&
		-\frac{2B}{\beta}(\phi(r))^{\frac{\alpha}{\beta}}  H^{\frac{1}{\beta}-1}
		\frac{\Phi}{\Phi-a} |A|^2-
		\frac{2Bn}{\beta} H^{\frac{1}{\beta}-1}(\phi(r))^{\frac{\alpha}{\beta}} \frac{\gamma u-\Phi}{\Phi-a} \\
		\nonumber
		&+ 
		\frac{4B^2-2B}{\beta} \phi^{\frac{\alpha}{\beta}} H^{\frac{1}{\beta}-1}  \frac{|\nabla\Phi|^2}{(\Phi-a)^2}\\
		=&I_1+\cdots+I_7,
	\end{align}
where $I_i$ denotes the $i$th line on the right hand side of \eqref{L.Q.temp1}.
	
	Since the eigenvalues of $ (h_{i}^{j}) $ are $ \kappa_{1},\cdots,\kappa_{n} $, and $ |A| = \sqrt{(\kappa_{1})^2+\cdots+(\kappa_{n})^2} \geq |\kappa_{i}|,\forall i $, we have
	$$ -|A|g^{ij} \leq h^{ij} \leq |A|g^{ij}, $$
	Without loss of generality, we may assume the local coordinate $ \{x_1,\cdots,x_n\} $ is orthonormal at the spatial maximum point of $ Q $, thus
	$$ \nabla_{i}r = \frac{\left\langle X_i,V\right\rangle}{\phi} \leq 1, $$
	which implies that
	$$ |\nabla r|^2 \leq n. $$
	So we have
	\begin{align}
		\label{tech1}
		&\left|\frac{h^{ij}}{|A|}(\nabla_{i}r)(\nabla_{j}r)\right| \leq
		|\nabla r|^2 \leq n,\\
		\label{tech2}
		&\left|\frac{h^{ij}}{|A|}(\nabla_{i}F)(\nabla_{j}F)\right| \leq
		|\nabla F|^2,\\
		\label{tech3}
		&\left|\frac{h^{ij}}{|A|}(\nabla_{i}r)(\nabla_{j}F)\right| \leq
		|\nabla r| |\nabla F| \leq
		\sqrt{n}|\nabla F|.
	\end{align}

	Using the Cauchy inequality, it is easy to see
	$$	\frac{\left\langle\nabla|A|^{2}, \nabla \Phi\right\rangle}{(\Phi-a)}
	=
	2 \nabla_{l} h_{i j}\left(h_{i j} \frac{\nabla_{l} \Phi}{\Phi-a}\right)
	\leq
	|\nabla A|^{2}+|A|^{2} \frac{|\nabla \Phi|^{2}}{(\Phi-a)^{2}}.$$
	By \eqref{L.Q.condition},
	$$	\frac{\nabla_{i}|A|^2}{|A|^2} = 2B\frac{\nabla_{i}\Phi}{\Phi-a},$$
	thus
	$$	\frac{|\nabla A|^{2}}{|A|^2}+ \frac{|\nabla \Phi|^{2}}{(\Phi-a)^{2}}
	\geq
	\frac{\left\langle\nabla|A|^{2}, \nabla \Phi\right\rangle}{|A|^2(\Phi-a)}
	=
	2B\frac{|\nabla\Phi|^2}{(\Phi-a)^2}.$$
	That is,
	\begin{equation}\label{tech4}
		\frac{|\nabla A|^{2}}{|A|^2}
		\geq
		(2B-1)\frac{|\nabla\Phi|^2}{(\Phi-a)^2}.
	\end{equation}
	
	By \eqref{tech2} and \eqref{tech4},
	we estimate the first line of \eqref{L.Q.temp1} as
	\begin{align}
		I_1	\leq
		\label{first.line}
		&- \frac{4B-2}{\beta} \phi^{\frac{\alpha}{\beta}} H^{\frac{1}{\beta}-1}
		\frac{|\nabla\Phi|^2}{(\Phi-a)^2}
		 +
		2 (1-\beta)\phi^{\frac{\alpha}{\beta}} H^{-\frac{1}{\beta}}
		\frac{1}{|A|}|\nabla F|^2.
	\end{align}
	
	By \eqref{tech1} and \eqref{tech3},
	we estimate the second and the third line of \eqref{L.Q.temp1} as
	\begin{align}
		I_2+I_3\leq & C + \frac{C}{|A|} + C|\nabla F|.\label{second.third.line}
	\end{align}
	for some constant $ C>0 $,
	by the boundness of $ r $, $ u $, and $ H $.
	
	For the rest four lines of \eqref{L.Q.temp1}, we have
	\begin{align}
		I_4+I_5+I_6+I_7		\leq &
		\frac{2}{\beta}\left(1-B\frac{\Phi}{\Phi-a}\right)\phi^{\frac{\alpha}{\beta}}  H^{\frac{1}{\beta}-1}
		|A|^2 \nonumber\\
		&+
		\frac{4B^2-2B}{\beta} \phi^{\frac{\alpha}{\beta}} H^{\frac{1}{\beta}-1}  \frac{|\nabla\Phi|^2}{(\Phi-a)^2} +
		C|A| +
		C + \frac{C}{|A|}.\label{rest.lines}
	\end{align}
	for some constant $ C>0 $.
	
	Combining \eqref{first.line}, \eqref{second.third.line} and \eqref{rest.lines}, and assuming that $ |A| $ is large enough, we get
	\begin{align}\label{L.Q.temp2}
		\mathcal{L}Q
		\leq &
		\frac{2}{\beta}\left(1-B\frac{\Phi}{\Phi-a}\right)\phi^{\frac{\alpha}{\beta}}  H^{\frac{1}{\beta}-1}
		|A|^2 \nonumber\\
&+
		\frac{2}{\beta} \left( 2B^2- 3B + 1 \right)
		\phi^{\frac{\alpha}{\beta}} H^{\frac{1}{\beta}-1}
		\frac{|\nabla \Phi|^2}{(\Phi-a)^2} +
		C|A| +
		C |\nabla F| + C,
	\end{align}
	for some constant $ C>0 $ at the spatial maximum point of $Q$.
	
	For the first term in \eqref{L.Q.temp2}, since
	$$	1 - B\frac{\Phi}{\Phi-a}
	= 1 - \frac{1- \frac{a}{2c}}{1-\frac{a}{\Phi}}
	\leq 1 - \frac{1- \frac{a}{2c}}{1-\frac{a}{c}}
	= - \frac{\frac{a}{2c}}{1-\frac{a}{c}},$$
	and $ (\phi(r))^{\frac{\alpha}{\beta}}H^{\frac{1}{\beta}-1} \geq c' $ for some positive constant $ c' $,
	we have
	\begin{equation}\label{first.term}
		\frac{2}{\beta}\left( 1 - B\frac{\Phi}{\Phi-a} \right) (\phi(r))^{\frac{\alpha}{\beta}}H^{\frac{1}{\beta}-1}|A|^2 \leq -c_{1}|A|^2,
	\end{equation}
	for some constant $ c_{1}>0 $.
	
	For the second term in \eqref{L.Q.temp2},
	$$
	2B^2-3B+1 = (2B-1)(B-1) = - \frac{a}{2c}\left( 1-\frac{a}{c} \right),
	$$
	which is a negative constant, and
	\begin{align*}
		&|\nabla \Phi|^2
		= \left|\nabla\left( (\phi(r))^{\frac{\alpha}{\beta}}  F \right)\right|^2
		= \left|\frac{\alpha}{\beta}(\phi(r))^{\frac{\alpha}{\beta}-1}\phi'(r)F\nabla r + (\phi(r))^{\frac{\alpha}{\beta}}\nabla F\right|^2\\
		= &\left(\frac{\alpha}{\beta}\right)^2(\phi(r))^{2\left(\frac{\alpha}{\beta}-1\right)}(\phi'(r))^2 F^2 |\nabla r|^2 +
		\frac{2\alpha}{\beta}(\phi(r))^{\frac{2\alpha}{\beta}-1}\phi'(r)F
		\left\langle \nabla r, \nabla F \right\rangle +
		(\phi(r))^{\frac{2\alpha}{\beta}} |\nabla F|^2 \\
		\geq & c'|\nabla F|^2 - C|\nabla F|,
	\end{align*}
	for some positive constant $ c', C $.
	Thus
	\begin{equation}\label{second.term}
		\frac{2}{\beta} \left( 2B^2- 3B + 1 \right)
		\phi^{\frac{\alpha}{\beta}} H^{\frac{1}{\beta}-1}
		\frac{|\nabla \Phi|^2}{(\Phi-a)^2}
		\leq -c_{2}|\nabla F|^2 + C |\nabla F|,
	\end{equation}
	for some positive constant $ c_2,C $.
	
	Plugging \eqref{first.term} and \eqref{second.term} into \eqref{L.Q.temp2}, we get
	\begin{equation}
		\mathcal{L}Q \leq
		-c_{1}|A|^2 - c_{2}|\nabla F|^2 + C|A| + C|\nabla F| + C \leq
		-c_{1}|A|^2 + C|A| + C.
	\end{equation}
	Since $ |A| $ tends to $ +\infty $ as $ Q $ tends to $ +\infty $, if $ Q $ is large enough, $ \mathcal{L}Q\leq 0 $, thus we prove the boundness of $ Q $, which implies the boundness of $ |A| $. So the principal curvatures $ \kappa_{i}(i=1,\cdots,n) $ are bounded.
\end{proof}

\section{$C^2$ estimate for uniformly convex solutions}\label{sec:6}
In this section, we consider uniformly convex solution of the flow \eqref{flow}. We will prove that if the initial hypersurface is uniformly convex, then there exists a positive uniform lower bound of principal curvatures.

To estimate the lower bound of the principal curvatures, we need the following lemma.
\begin{lem}[see \cite{L-S-W-2020}, \cite{Urb-1991}]\label{inverse.concave}
	Let $\left\{\tilde{h}^{i j}\right\}$ be the inverse of $\left\{h_{i j}\right\}$. Then $\left\{\tilde{h}_{i}^{j}\right\}$ represents the inverse of Weingarten map. If $\left\{h_{i}^{j}\right\}>0$, we have
	$$
	\left({G}^{p q, l m}+2 {G}^{p m} \tilde{h}_{q}^{l}\right) \eta_{p}^{q} \eta_{l}^{m} \geq 2 G^{-1}\left({G}^{p q} \eta_{p}^{q}\right)^{2}.
	$$
	for any tensor $\left\{\eta_{p}^{q}\right\}$, where $G=\sigma_{k}^{\frac{1}{k}}\left(h_{i}^{j}\right)$.
\end{lem}
Using \lemref{inverse.concave} now we can prove
\begin{prop}\label{C2estimate5}
	Let $X(\cdot, t)$ be a smooth, closed and uniformly convex solution to the flow \eqref{flow} for $t \in[0, T)$, which enclosed the origin. If $\alpha > k+\beta$, there exists a positive constant $C$ depending only on $\alpha$ and $M_{0}$, such that the principal curvatures of $X(\cdot, t)$ satisfy
	$$
	\frac{1}{C} \leq \kappa_{i}(\cdot, t) \leq C, \quad \forall t \in[0, T) \text { and } i=1,2, \cdots, n.
	$$
\end{prop}
\begin{proof}
	As we have already proved the upper bound of $\sigma_k$, to derive the uniform bounds on the principal curvatures, it suffices to prove the uniform positive lower bound on the principal curvatures
 \begin{equation*}
   \kappa_i(x,t)\geq \frac 1C>0.
 \end{equation*}
 Let $\lambda(x, t)$ denote the maximal principal radii (that is, the reciprocal  of the minimal principal curvature) at $X(x, t)$. In the following, we will prove the upper bound of $\lambda(x,t)$.

 Let $ (x_0,t_0) $ denote the point where $ \lambda(x,t) $ attains its spatial maximum at time $ t=t_0 $. Without loss of generality, we choose a local normal coordinate $\left\{x_{1}, \cdots, x_{n}\right\}$ near $P_{0}=\left(x_{0}, t_{0}\right)$ such that the second fundamental form $h_{ij}$ is diagonal at $P_0$. 
	For now, the equation about evolution of geometric quantities, that is, \lemref{evolve.quantities} and \corref{evolve.quantities1} still holds.

	Furthermore, we can assume $\left.\partial_{1}\right|_{\left(x_{0}, t_{0}\right)}$ is an eigenvector with respect to $\lambda\left(x_{0}, t_{0}\right)$, i.e., $\lambda\left(x_{0}, t_{0}\right)=\tilde{h}_{1}^{1}\left(x_{0}, t_{0}\right) .$ Using this coordinate system, we can calculate the evolution of $\tilde{h}_{1}^{1}$ at $\left(x_{0}, t_{0}\right)$.
	
	From
		\begin{align*}
			\frac{\partial}{\partial t} \tilde{h}_{1}^{1}&=-\left(\tilde{h}_{1}^{1}\right)^{2} \partial_{t} h_{1}^{1}, \\
			\nabla_{i} \tilde{h}_{1}^{1}=\frac{\partial \tilde{h}_{1}^{1}}{\partial h_{p}^{q}} \nabla_{i} h_{p}^{q}&=-\tilde{h}_{1}^{p} \tilde{h}_{q}^{1} \nabla_{i} h_{p}{ }^{q}=-\left(\tilde{h}_{1}^{1}\right)^{2} \nabla_{i} h_{11},
		\end{align*}
	and
		\begin{align*}
			\nabla_{j} \nabla_{i} \tilde{h}_{1}^{1} &=\nabla_{j}\left(-\tilde{h}_{1}^{p} \tilde{h}_{q}^{1} \nabla_{i} h_{p}^{q}\right) \\
			&=-\nabla_{j} \tilde{h}_{1}^{p} \tilde{h}_{q}^{1} \nabla_{i} h_{p}^{q}-\tilde{h}_{1}^{p} \nabla_{j} \tilde{h}_{q}^{1} \nabla_{i} h_{p}^{q}-\tilde{h}_{1} p \tilde{h}_{q}^{1} \nabla_{j} \nabla_{i} h_{p}^{q} \\
			&=\tilde{h}_{1}^{r} \tilde{h}_{l}^{p} \tilde{h}_{q}^{1} \nabla_{j} h_{r}^{l} \nabla_{i} h_{p}^{q}+\tilde{h}_{1}^{p} \tilde{h}_{r}^{1} \tilde{h}_{q}^{s} \nabla_{j} h_{s}^{r} \nabla_{i} h_{p}^{q}-\tilde{h}_{1}^{p} \tilde{h}_{q}^{1} \nabla_{j} \nabla_{i} h_{p}^{q} \\
			&=-\left(\tilde{h}_{1}^{1}\right)^{2} \nabla_{j} \nabla_{i} h_{1}^{1}+2\left(\tilde{h}_{1}^{1}\right)^{2} \tilde{h}^{p l} \nabla_{i} h_{1 p} \nabla_{j} h_{1 l},
		\end{align*}
	we get
	\begin{align}
		\nonumber
		\mathcal{L} \tilde{h}_{1}^{1}
		& = \p_{t}\tilde{h}_{1}^{1} - (\phi(r))^{\frac{\alpha}{\beta}}\dot{F}^{ij}\nabla_{i}\nabla_{j}\tilde{h}_{1}^{1}\\
		\label{L.h11.temp1}
		& = -\left(\tilde{h}_{1}^{1}\right)^{2} \partial_{t} h_{1}^{1} +
		(\phi(r))^{\frac{\alpha}{\beta}}\dot{F}^{ij}
		\left(\tilde{h}_{1}^{1}\right)^{2}
		\nabla_{i} \nabla_{j} h_{1}^{1}\nonumber\\
		&\quad  -
		2 (\phi(r))^{\frac{\alpha}{\beta}}\dot{F}^{ij}\left(\tilde{h}_{1}^{1}\right)^{2} \tilde{h}^{p q} \nabla_{i} h_{1 p} \nabla_{j} h_{1 q}.
	\end{align}
	By \corref{evolve.quantities1},
	\begin{equation}
		\begin{aligned}\label{p.t.h11}
			\p_{t} h_{1}^{1}
			= &\nabla_{1}\nabla^{1}\Phi+\Phi \left(h_{1}^{1}\right)^2 -
			\gamma\phi'(r)h_{1}^{1} + (\gamma u - \Phi).
		\end{aligned}
	\end{equation}
	By \eqref{bigresult},
	\begin{align}
		\nabla_{1}\nabla^{1}\Phi
		\nonumber
		=&(\phi(r))^{\frac{\alpha}{\beta}}\dot{F}^{pq}\nabla_{p}\nabla_{q}h_{1}^{1}
		+ (\phi(r))^{\frac{\alpha}{\beta}}\ddot{F}^{pq,rs}h_{rs,1}h_{pq,1} \\
		\nonumber
		&+\left(\frac{\alpha}{\beta}\right)^2\frac{(\phi'(r))^2}{(\phi(r))^2}
		(\nabla_{1}r)^2
		\Phi -
		\frac{\alpha}{\beta}\frac{1}{(\phi(r))^2}(\nabla_{1}r)^2\Phi -
		\frac{\alpha}{\beta}\frac{\phi'(r)}{(\phi(r))^2}uh_{1}^{1}\Phi \\
		\nonumber
		&+\frac{\alpha}{\beta}\frac{(\phi'(r))^2}{(\phi(r))^2}
		\left(
		1-(\nabla_{1}r)^2
		\right)\Phi
		+2\frac{\alpha}{\beta}\phi'(r) (\phi(r))^{\frac{\alpha}{\beta}-1}
		(\nabla_{1}r)(\nabla_{1}F)
		\\
		\nonumber
		&+
		\frac{1}{\beta}(\phi(r))^{\frac{\alpha}{\beta}}(\sigma_{k})^{\frac{1}{\beta}-1}
		\left(\sigma_1\sigma_{k}-(k+1)\sigma_{k+1}\right)h_{1}^{1}
		-\frac{k}{\beta}\Phi\left(h_{1}^{1}\right)^2 \\
		\label{nabla.11.Phi}
		&-\frac{k}{\beta}\Phi +
		\frac{n-k+1}{\beta}(\sigma_{k})^{\frac{1}{\beta}-1}(\phi(r))^{\frac{\alpha}{\beta}}\sigma_{k-1}h_{1}^{1}.
	\end{align}
	Plugging \eqref{p.t.h11} and \eqref{nabla.11.Phi} into \eqref{L.h11.temp1}, we have
	\begin{align}
		\nonumber
		&
		\mathcal{L}\tilde{h}_{1}^{1} \\
		\nonumber
		= &
		-\left(\tilde{h}_{1}^{1}\right)^2
		\left[
		\phi^{\frac{\alpha}{\beta}}\ddot{F}^{pq,rs}h_{rs,1}h_{pq,1}+
		\left(\frac{\alpha}{\beta}\right)^2\frac{(\phi')^2}{\phi^2}(\nabla_{1}r)^2
		\Phi -
		\frac{\alpha}{\beta}\frac{1}{\phi^2}(\nabla_{1}r)^2\Phi
		\right. \\
		\nonumber
		&-
		\frac{\alpha}{\beta}\frac{\phi'}{\phi^2}uh_{1}^{1}\Phi
		+
		\frac{\alpha}{\beta}\frac{(\phi')^2}{\phi^2}
		\left(
		1 - (\nabla_{1}r)^2
		\right)\Phi
		+
		2\frac{\alpha}{\beta}\phi' \phi^{\frac{\alpha}{\beta}-1}
		(\nabla_{1}r)(\nabla_{1}F) \\
		\nonumber &+
		\frac{1}{\beta}\phi^{\frac{\alpha}{\beta}}(\sigma_{k})^{\frac{1}{\beta}-1}
		\left(\sigma_1\sigma_{k}-(k+1)\sigma_{k+1}\right)h_{1}^{1}
		-
		\frac{k}{\beta}\Phi \left(h_{1}^{1}\right)^2
		- \frac{k}{\beta}\Phi \\
		\nonumber
		&\left.+
		\frac{n-k+1}{\beta}\phi^{\frac{\alpha}{\beta}}(\sigma_{k})^{\frac{1}{\beta}-1}
		\sigma_{k-1}h_{1}^{1} +
		\Phi \left(h_{1}^{1}\right)^2 -
		\gamma\phi'h_{1}^{1} +
		(\gamma u - \Phi)
		\right] \\
		\nonumber
		&-
		2\phi^{\frac{\alpha}{\beta}}\dot{F}^{ij}\left(\tilde{h}_{1}^{1}\right)^{2} \tilde{h}^{p q} \nabla_{i} h_{1 p} \nabla_{j} h_{1 q} \\
		\nonumber
		= &
		-\left(\tilde{h}_{1}^{1}\right)^2
		\left[
		\left(\frac{\alpha}{\beta}\right)^2\frac{(\phi')^2}{\phi^2}(\nabla_{1}r)^2
		\Phi -
		\frac{\alpha}{\beta}\frac{1}{\phi^2}(\nabla_{1}r)^2\Phi -
		\frac{\alpha}{\beta}\frac{\phi'}{\phi^2}uh_{1}^{1}\Phi
		\right. \\
		\nonumber
		&+
		\left.
		\frac{\alpha}{\beta}\frac{(\phi')^2}{\phi^2}
		\left(
		1 - (\nabla_{1}r)^2
		\right)\Phi +
		2\frac{\alpha}{\beta}\phi' \phi^{\frac{\alpha}{\beta}-1}
		(\nabla_{1}r)(\nabla_{1}F)
		\right]
		-
		\left(\tilde{h}_{1}^{1}\right)^2\left(\gamma u - \frac{k+\beta}{\beta}\Phi\right) \\
		\nonumber
		& -
		\left(\tilde{h}_{1}^{1}\right)^2 \phi^{\frac{\alpha}{\beta}}
		\left[
		\ddot{F}^{pq,rs}h_{rs,1}h_{pq,1} +
		2\dot{F}^{ij}
		\tilde{h}^{p q} h_{1 p,i} h_{1 q,j}
		\right] -
		\frac{n-k+1}{\beta}\phi^{\frac{\alpha}{\beta}}(\sigma_{k})^{\frac{1}{\beta}-1}\sigma_{k-1}\tilde{h}_{1}^{1}
		\\
		\label{L.h11.temp2}
		& -
		\frac{1}{\beta}\phi^{\frac{\alpha}{\beta}}(\sigma_{k})^{\frac{1}{\beta}-1}
		\left(\sigma_1\sigma_{k}-(k+1)\sigma_{k+1}\right)\tilde{h}_{1}^{1} +
		(\frac{k}{\beta}-1)\Phi +
		\gamma\phi'\tilde{h}_{1}^{1}.
	\end{align}
	Then we denote $G=\sigma_{k}^{\frac{1}{k}}$, thus $F=G^{\frac{k}{\beta}}$. We have
	\begin{align}
		\nonumber
		\dot{F}^{p q}
		&
		=\frac{k}{\beta} G^{\frac{k}{\beta}-1} \dot{G}^{p q} \\
		\label{sigma.G}
		\ddot{F}^{p q, r s}
		&
		=\frac{k}{\beta} G^{\frac{k}{\beta}-1} \ddot{G}^{p q, r s}+\frac{k}{\beta}\left(\frac{k}{\beta}-1\right) G^{\frac{k}{\beta}-2} \dot{G}^{p q} \dot{G}^{r s}
	\end{align}
	Thus for the first term in the third line of \eqref{L.h11.temp2}, we have
	\begin{align}
		\nonumber
		&\ddot{F}^{pq,rs}h_{rs,1}h_{pq,1} +
		2\dot{F}^{ij}
		\tilde{h}^{p q} h_{1 p,i} h_{1 q,j} \\
		\nonumber
		=&\frac{k}{\beta} G^{\frac{k}{\beta}-1}
		\left(
		G^{pq,rs}h_{rs,1}h_{pq,1} +
		2 G^{ij}
		\tilde{h}^{p q} h_{1 p,i} h_{1 q,j}
		\right) +
		\frac{k}{\beta}\left(\frac{k}{\beta}-1\right) G^{\frac{k}{\beta}-2} \dot{G}^{p q} \dot{G}^{r s}h_{pq,1}h_{rs,1} \\
		\label{that.term1}
		=&\frac{k}{\beta} G^{\frac{k}{\beta}-1}
		\left(
		G^{pq,rs}h_{rs,1}h_{pq,1} +
		2 G^{ij}
		\tilde{h}^{p q} h_{1 p,i} h_{1 q,j}
		\right) +
		\frac{k}{\beta}\left(\frac{k}{\beta}-1\right) G^{\frac{k}{\beta}-2} \left(\nabla_{1}G\right)^2.
	\end{align}
	By \lemref{inverse.concave} and Codazzi equation, we have
	\begin{equation}\label{that.term2}
		G^{pq,rs}h_{rs,1}h_{pq,1} +
		2 G^{ij}
		\tilde{h}^{p q} h_{1 p,i} h_{1 q,j}
		\geq
		2 G^{-1} \left( G^{pq} h_{pq,1} \right)^2
		=
		2 G^{-1} \left( \nabla_{1}G \right)^2.
	\end{equation}
	Combining \eqref{that.term1} and \eqref{that.term2}, we get
	\begin{align}
		\nonumber
		\left(\tilde{h}_{1}^{1}\right)^2 \phi^{\frac{\alpha}{\beta}}
		\left[\ddot{F}^{pq,rs}h_{rs,1}h_{pq,1} +
		2\dot{F}^{ij}
		\tilde{h}^{p q} h_{1 p,i} h_{1 q,j} \right]
		\geq
		&
		\frac{k}{\beta}\left(\frac{k}{\beta}+1\right)
		\left(\tilde{h}_{1}^{1}\right)^2 \phi^{\frac{\alpha}{\beta}}G^{\frac{k}{\beta}}G^{-2}
		\left(\nabla_1 G\right)^2\\
		\label{that.term}
		=&\frac{k}{\beta}\left(\frac{k}{\beta}+1\right)
		\left(\tilde{h}_{1}^{1}\right)^2\Phi G^{-2}\left(\nabla_1 G\right)^2
	\end{align}
	Substituting \eqref{that.term} into \eqref{L.h11.temp2} and observing that $\phi'^2-\phi^2=1$, we have
	\begin{align*}\label{L.h11.temp3}
		&
		\mathcal{L}\tilde{h}_{1}^{1} \\
		\leq &
		-\left(\tilde{h}_{1}^{1}\right)^2
		\left[
		\left(\frac{\alpha}{\beta}\right)^2\frac{(\phi')^2}{\phi^2}(\nabla_{1}r)^2
		\Phi -
		\frac{\alpha}{\beta}\frac{(\phi')^2}{\phi^2}(\nabla_{1}r)^2\Phi +
		\frac{\alpha}{\beta}(\nabla_{1}r)^2\Phi
		\right. \\
		& +
		\left.
		\frac{\alpha}{\beta}\frac{(\phi')^2}{\phi^2}
		\left(
		1 - (\nabla_{1}r)^2
		\right)\Phi +
		2\frac{k}{\beta}\frac{\alpha}{\beta}\frac{\phi'}{\phi}G^{-1}
		(\nabla_{1}r)(\nabla_{1}G)
		\Phi
		\right] \\
		& -
		\frac{k}{\beta} \left(\frac{k}{\beta}+1\right) \left(\tilde{h}_{1}^{1}\right)^2
		G^{-2} \left( \nabla_{1} G \right)^2 \Phi
		-
		\left(\tilde{h}_{1}^{1}\right)^2\left(\gamma u - \frac{k+\beta}{\beta}\Phi\right)\\
		&-
		\frac{n-k+1}{\beta}\phi^{\frac{\alpha}{\beta}}(\sigma_{k})^{\frac{1}{\beta}-1}\sigma_{k-1}\tilde{h}_{1}^{1}
		-
		\frac{1}{\beta}\phi^{\frac{\alpha}{\beta}}(\sigma_{k})^{\frac{1}{\beta}-1}
		\left(\sigma_1\sigma_{k}-(k+1)\sigma_{k+1}\right)\tilde{h}_{1}^{1}\\
		& +
		(\frac{k}{\beta}-1)\Phi +
		\gamma\phi'\tilde{h}_{1}^{1} + \frac{\alpha}{\beta} \frac{\phi^{\prime}}{\phi^{2}} u \tilde{h}_{1}^{1} \Phi \\
		= &
		-\left(\tilde{h}_{1}^{1}\right)^2
		\frac{\alpha}{\beta}\frac{(\phi')^2}{\phi^2}
		\left(
		1 - (\nabla_{1}r)^2
		\right)\Phi -
		\left(\tilde{h}_{1}^{1}\right)^2\left(\gamma u - \frac{k+\beta}{\beta}\Phi\right) -
		\left(\tilde{h}_{1}^{1}\right)^2 \frac{\alpha}{\beta}(\nabla_1 r)^2\Phi \\
		& -
		\left(\tilde{h}_{1}^{1}\right)^2\Phi
		\left[
		\frac{\alpha}{\beta}\left(\frac{\alpha}{\beta}-1\right)\frac{(\phi')^2}{\phi^2}
		\left( \nabla_{1} r \right)^2 +
		2\frac{\alpha}{\beta} \frac{k}{\beta} \frac{\phi'}{\phi}\left( \nabla_{1} r \right)
		G^{-1}\left( \nabla_{1} G \right) +
		\frac{k}{\beta}\left(\frac{k}{\beta}+1\right) G^{-2}\left( \nabla_{1} G \right)^2
		\right] \\
		& -
		\frac{n-k+1}{\beta}\phi^{\frac{\alpha}{\beta}}(\sigma_{k})^{\frac{1}{\beta}-1}\sigma_{k-1}\tilde{h}_{1}^{1}
		-
		\frac{1}{\beta}\phi^{\frac{\alpha}{\beta}}(\sigma_{k})^{\frac{1}{\beta}-1}
		\left(\sigma_1\sigma_{k}-(k+1)\sigma_{k+1}\right)\tilde{h}_{1}^{1}\\
		& +
		(\frac{k}{\beta}-1)\Phi +
		\gamma\phi'\tilde{h}_{1}^{1} + \frac{\alpha}{\beta} \frac{\phi^{\prime}}{\phi^{2}} u \tilde{h}_{1}^{1} \Phi \\
		= &
		-\left(\tilde{h}_{1}^{1}\right)^2
		\frac{\alpha}{\beta}\frac{(\phi')^2}{\phi^2}
		\left(
		1 - (\nabla_{1}r)^2
		\right)\Phi -
		\left(\tilde{h}_{1}^{1}\right)^2
		\left(\gamma u - \frac{k+\beta}{\beta}\Phi\right)-
		\left(\tilde{h}_{1}^{1}\right)^2 \frac{\alpha}{\beta}(\nabla_1 r)^2\Phi \\
		& +
		\frac{\alpha}{\beta}
		\left(
		\frac{k\alpha}{\beta^2\left(\frac{k}{\beta}+1\right)} - (\frac{\alpha}{\beta}-1)
		\right)
		\left(\tilde{h}_{1}^{1}\right)^2
		\frac{(\phi')^2}{\phi^2}
		\left( \nabla_{1} r \right)^2\Phi
		\\
		& -
		\left(\tilde{h}_{1}^{1}\right)^2 \frac{k}{\beta} \Phi
		\left[
		\frac{\alpha^2}{\beta^2 \left(\frac{k}{\beta}+1\right)}\frac{(\phi')^2}{\phi^2}
		\left( \nabla_{1} r \right)^2 +
		2\frac{\alpha}{\beta}  \frac{\phi'}{\phi}\left( \nabla_{1} r \right)
		G^{-1}\left( \nabla_{1} G \right) +
		\left(\frac{k}{\beta}+1\right) G^{-2}\left( \nabla_{1} G \right)^2
		\right] \\
		& -
		\frac{n-k+1}{\beta}\phi^{\frac{\alpha}{\beta}}(\sigma_{k})^{\frac{1}{\beta}-1}\sigma_{k-1}\tilde{h}_{1}^{1}
		-
		\frac{1}{\beta}\phi^{\frac{\alpha}{\beta}}(\sigma_{k})^{\frac{1}{\beta}-1}
		\left(\sigma_1\sigma_{k}-(k+1)\sigma_{k+1}\right)\tilde{h}_{1}^{1}\\
		& +
		(\frac{k}{\beta}-1)\Phi +
		\gamma\phi'\tilde{h}_{1}^{1} + \frac{\alpha}{\beta} \frac{\phi^{\prime}}{\phi^{2}} u \tilde{h}_{1}^{1} \Phi.
	\end{align*}
	For these coefficients, we have
	\begin{equation}\label{first.coefficient}
		\frac{k\alpha}{\beta^2\left(\frac{k}{\beta}+1\right)} - (\frac{\alpha}{\beta}-1) = -\frac{\alpha - k - \beta}{k+\beta} < 0,
	\end{equation}
	and
	\begin{align}
		\nonumber
		&\frac{\alpha^2}{\beta^2 \left(\frac{k}{\beta}+1\right)}\frac{(\phi')^2}{\phi^2}
		\left( \nabla_{1} r \right)^2 +
		2\frac{\alpha}{\beta}  \frac{\phi'}{\phi}\left( \nabla_{1} r \right)
		G^{-1}\left( \nabla_{1} G \right) +
		\left(\frac{k}{\beta}+1\right) G^{-2}\left( \nabla_{1} G \right)^2 \\
		\label{second.coefficient}
		= &
		\left(\frac{k}{\beta}+1\right)
		\left( \frac{\alpha}{k+\beta}\frac{\phi'}{\phi}\nabla_{1}r +
		G^{-1}\nabla_{1}G
		\right)^2
		\geq 0,
	\end{align}
	Plugging \eqref{first.coefficient} and \eqref{second.coefficient} into the inequality of $ \mathcal{L}\tilde{h}_{1}^{1} $ above, and use the boundness of $ r $, $ u $ and $ \Phi $, we have
	\begin{align}\label{L.h11.temp4}
		\nonumber
			\mathcal{L}\tilde{h}_{1}^{1}
			\leq
			&\left(\tilde{h}_{1}^{1}\right)^2
			\left(
			- \frac{\alpha}{\beta}\frac{(\phi')^2}{\phi^2}\Phi
			+ \frac{\alpha}{\beta}\frac{(\phi')^2}{\phi^2}(\nabla_1r)^2\Phi
			- \frac{\alpha}{\beta}(\nabla_1r)^2\Phi
			- \gamma u + \frac{k+\beta}{\beta}\Phi
			\right)\\
			&+ C \tilde{h}_{1}^{1} + C,
	\end{align}
	for some constant $ C>0 $.
	
	Since
	$$
	\nabla_1r = \frac{\left\langle X_1,V \right\rangle}{\phi} \leq \frac{|X_1||V|}{|V|} = 1,
	$$
	we can estimate the coefficient of $ \left(\tilde{h}_{1}^{1}\right)^2 $ in \eqref{L.h11.temp4} as following
	\begin{align}
		\nonumber
		&- \frac{\alpha}{\beta}\frac{(\phi')^2}{\phi^2}\Phi
		+ \frac{\alpha}{\beta}\frac{(\phi')^2}{\phi^2}(\nabla_1r)^2\Phi
		- \frac{\alpha}{\beta}(\nabla_1r)^2\Phi
		- \gamma u + \frac{k+\beta}{\beta}\Phi \\
		=
		\nonumber
		&- \frac{\alpha}{\beta}\frac{(\phi')^2}{\phi^2}\Phi
		+ \frac{\alpha}{\beta}\frac{1}{\phi^2}(\nabla_1r)^2\Phi
		+ \frac{k+\beta}{\beta}\Phi
		- \gamma u \\
		\leq
		\nonumber
		& - \frac{\alpha}{\beta}\frac{(\phi')^2}{\phi^2}\Phi
		+ \frac{\alpha}{\beta}\frac{1}{\phi^2}\Phi
		+ \frac{k+\beta}{\beta}\Phi
		- \gamma u \\
		\label{delta}
		=
		&-\frac{\alpha-k-\beta}{\beta}\Phi - \gamma u
		\leq
		-\delta,
	\end{align}
	for a constant $ \delta>0 $, by the lower bounds of $ u $ and $ \Phi $.
	
	Plugging \eqref{delta} into \eqref{L.h11.temp4}, we get
	\begin{equation}
		\mathcal{L}\tilde{h}_{1}^{1}
		\leq
		-\delta\left(\tilde{h}_{1}^{1}\right)^2 +
		C \tilde{h}_{1}^{1} + C,
	\end{equation}
	for some positive constants $ \delta $, and $ C $.
	Thus when $ \tilde{h}_{1}^{1} $ is large enough, $ \mathcal{L}\tilde{h}_{1}^{1} \leq 0 $, so $ \tilde{h}_{1}^{1} $ is bounded from above. That is, the principal radii is bounded from above, which implies the principal curvature is bounded from below. Moreover, by \propref{C2estimate2}, $ \sigma_{k} $ is bounded from above, thus the principal curvature is bounded from both above and below.
\end{proof}

\section{Long time existence and convergence}\label{sec:7}

Now we have obtained the a priori estimates of flow \eqref{flow}, with the initial hypersurfaces mentioned in Theorem \ref{k-convex.converge} and Theorem \ref{u.convex.converge}. From a priori estimates, these flows have short time existence. Using the $C^{2}$ estimates given in \propref{C2estimate3}, \propref{C2estimate5}, we can get the $C^{2, \lambda}$ estimate of the scalar equation \eqref{hradial.eq} by using Theorem 6 in Andrews \cite{Andrews-2004}, where $\lambda\in (0,1)$. Note that as $0<\beta\leq 1$ the equation \eqref{hradial.eq} is in general not concave with respect to the second spatial derivatives and the result of Krylov \cite{Krylov-1987} can not be applied directly. Then the parabolic Schauder theory \cite{Lieb96} implies the $C^{k,\lambda}$ estimates for all $k\geq 2$. Hence, we get the long time existence of these flows.
\begin{prop}\label{long.time.exist}
	The smooth solution of the flow \eqref{flow} with the initial hypersuface mentioned in Theorem \ref{k-convex.converge} and Theorem \ref{u.convex.converge} exists for all time $t \in[0,+\infty)$.
\end{prop}

We now complete the proofs of Theorem \ref{k-convex.converge} and Theorem \ref{u.convex.converge}.
\begin{proof}
	By \eqref{p.nabla.varphi} in the proof of \lemref{C1estimate},
	\begin{equation}
		\p_{t}\left(\frac{1}{2} |\bar{\nabla} \varphi|^{2} \right) \leq
		-\frac{\alpha-k-\beta}{\beta}(\sinh(-\varphi))^{-\frac{\alpha}{\beta}}\cosh(-\varphi)\sigma_{k}^{\frac{1}{\beta}}\sqrt{1+|\bar{\nabla} \varphi|^2}|\bar{\nabla}\varphi|^2.
	\end{equation}
	From \lemref{C0estimate} and \propref{C2estimate1} we know there exists a positive constant $ c $, such that
	\begin{equation}
		\p_{t}\left( |\bar{\nabla} \varphi|^{2} \right) \leq
		-c |\bar{\nabla}\varphi|^2.
	\end{equation}
	thus  $\max_{\mathbb{S}^{n}}|\bar{\nabla}\varphi|$ converges to $ 0 $ exponentially fast, which implies $\max_{\mathbb{S}^{n}}|\bar{\nabla} r|$ converges to $ 0 $ exponentially fast.
	
	The long time existence of the flow \eqref{flow} has been shown in \propref{long.time.exist}. Now we prove that the hypersurfaces converge to a sphere along the flow \eqref{flow}. As  $\max_{\mathbb{S}^{n}}|\bar{\nabla} r|$ converges to 0 exponentially fast when $\alpha > \beta+k$ and $n \geq 2$, we know $\max _{\mathbb{S}^{n}}\left|\bar{\nabla}^{l} r\right|$ decays exponentially to 0 for any $l \geq 1$ from the interpolation inequality and the a priori estimates we have made.
	When $\alpha>\beta+k$, we claim that the radial function $r$ converges to $ \hat{r} $ . Since $\MM_{0}$ is closed and star-shaped, it can be bounded by two spheres centred at the origin, we denote them as $\mathbb{S}^{n}\left(a_{1}\right)$ and $\mathbb{S}^{n}\left(a_{2}\right), a_{1}<a_{2}$. In the proof of \lemref{C0estimate}, we know the radial function of $M_{t}$ can be bounded by $a_{1}(t)$ and $a_{2}(t)$ by the well-known comparison principle, where
	\begin{equation}
		\frac{d a_{i}(t)}{d t} = -\gamma\sinh(a_{i}(t))\left( \eta(a_{i}(t)) - \eta(\hat{r}) \right), i =1,2,
	\end{equation}
	where $ \eta $ is the monotically increasing function $ \eta(R) = \sinh(R)^{\frac{\alpha-k-\beta}{\beta}} \cosh(R)^{\frac{k}{\beta}} $, by estimation of such ODEs, we see $ a_1(t) $ and $ a_2(t) $ tends to $ \hat{r} $ as $ t \rightarrow +\infty $, which implies that the radial function of $\MM_{t}$ converges to 1 in $C^{0}$ norm. Combining the estimates of $|\bar{\nabla} r|$, we know $\MM_{t}$ converges to $\mathbb{S}^{n}(\hat{r})$ exponentially fast.
\end{proof}

Finally, we conclude the paper with the following remark and question. 

\begin{rem}
By the same argument, the curvature bounds in Section \ref{sec:6} can also be obtained if we replace $\sigma_k^{1/k}$ by a general $1$-homogeneous symmetric curvature function $G(\kappa)$ which is inverse-concave and with dual $G_*$ vanishing on the boundary of the positive cone $\Gamma_n^+$, where $G_*(z_1,\cdots,z_n)=G(z_1^{-1},\cdots,z_n^{-1})^{-1}$. We focus $G=\sigma_k^{1/k}$ in this paper because of its relationship with the prescribed curvature measure problems in hyperbolic space. 
\end{rem}

\begin{ques}
Is it possible to obtain the curvature estimate for the $k$-convex solution of the flow \eqref{hflow} with $k=2,\cdots,n$?
\end{ques}

\bibliographystyle{amsplain}

\end{document}